\documentclass[10pt,leqno]{article}

\usepackage[lite]{amsrefs}
\usepackage{latexsym,enumerate,pstricks}

\usepackage[notref, notcite]{}
\usepackage{amssymb, amsfonts, eucal,amsmath,amsthm}

\newtheorem{theorem}{Theorem}[section]

\newtheorem{lemma}[theorem]{Lemma}
\newtheorem{proposition}[theorem]{Proposition}
\newtheorem{remark}[theorem]{Remark}
\newtheorem{definition}[theorem]{Definition}

\newtheorem{example}[theorem]{Example}
\newtheorem{corollary}[theorem]{Corollary}

\newcommand{\sect}[1]{\section{#1} \setcounter{equation}{0} }

\newcounter{ca}
\newcommand{\norm}[2]{\left\|#1\right\|_{#2}}

\newcommand{\dn}{\rho_n}
\newcommand{\dm}{\rho_m}
\newcommand{\dnx}{\dn(x)}
\newcommand{\dmx}{\dm(x)}

\newcommand{\ds}{\displaystyle}

 \newcommand{\ec}{\end{comment}}
\newcommand{\bc}{ \begin{comment}
 }
 \newcommand{\bsmall}{\begin{scriptsize} \vspace{1cm} \mbox{}\\
  $\downarrow\rule{\textwidth}{0.1mm}  \downarrow$
 \setlength{\baselineskip}{0.3mm}
 }
 \newcommand{\esmall}{\noindent $\uparrow \rule{\textwidth}{0.1mm}\uparrow$ \end{scriptsize}\vspace{1cm} }

\newcommand{\E}{{\mathcal E}}

  \newcommand{\II}{{\mathfrak J}}
 \newcommand{\I}{{\mathcal I}}

\newcommand{\ccc}{{\mathfrak  \vartheta}}

\newcommand{\A}{{\mathcal A}}

\newcommand{\D}{{\mathcal D}}

\newcommand{\andd}{\quad\mbox{\rm and}\quad}

\newcommand\e{{\varepsilon}}

\newcommand\w{{\omega}}
\newcommand{\Q}{{\mathcal Q}}
\newcommand{\Z}{{\mathcal Z}}

\newcommand{\pmin}{\mathrm{pmin}}

\def\be  {\begin{equation}}
\def\ee  {\end{equation}}
\def\ba  {\begin{eqnarray}}
\def\ea  {\end{eqnarray}}
\def\baa {\begin{eqnarray*}}
\def\eaa {\end{eqnarray*}}
\newenvironment{comment}[2]
{\bgroup\vspace{7pt}
\begin{tabular}{|p{5in}|}
\hline \qquad \bf \footnotesize Comment -- to be deleted in the final version \\
\hline
\quad\sl\footnotesize #1#2} {\\ \hline \end{tabular}
\vspace{7pt}\indent\egroup}

\def\updots{\mathinner{\mkern
1mu\raise 1pt \hbox{.}\mkern 2mu \mkern 2mu \raise
4pt\hbox{.}\mkern 1mu \raise 7pt\vbox {\kern 7 pt\hbox{.}}} }

\def \dist{\mathop{\rm dist}\nolimits}

\newcommand{\z}{\mathit{z}}

\newcommand{\C}{\mathbb C}

\newcommand{\W}{{\mathcal{W}}}

\newcommand{\R}{\mathbb R}

\newcommand{\N}{\mathbb N}

 \renewcommand{\a}{\alpha}
\renewcommand{\b}{\beta}
\newcommand{\ineq}[1]{(\ref{#1})}

\newcommand{\ie}{{\em i.e., }}
\newcommand{\eg}{{\em e.g. }}

\newcommand{\bpic}{
\begin{center}
}

\newcommand{\epic}{
\endpspicture
\end{center}
}

\newcommand{\st}{\;\; \big| \;\;}
\renewcommand{\L}{\mathbb{L}}
\newcommand{\Lp}{\L_p}

\newcommand{\Poly}{\Pi}
 \newcommand{\AC}{\mathrm{AC}}
  \newcommand{\loc}{\mathrm{loc}}

\newcommand{\Dom}{{\mathfrak{D}}}

 \newcommand{\newconst}{\sigma}

\newcommand{\thm}[1]{Theorem~\ref{#1}}
\newcommand{\lem}[1]{Lemma~\ref{#1}}
\newcommand{\lemp}[1]{Lemma~\ref{properties}(\ref{#1})}
\newcommand{\cor}[1]{Corollary~\ref{#1}}
\newcommand{\prop}[1]{Proposition~\ref{#1}}

\title{Polynomial  approximation with doubling weights having finitely many   zeros and singularities}

\author{
Kirill A.  Kopotun\thanks{Department of Mathematics, University of
Manitoba, Winnipeg, Manitoba, R3T 2N2, Canada ({\tt
kopotunk@cc.umanitoba.ca}). Supported by NSERC of Canada.}    }

\begin{document}

\maketitle

\begin{abstract}
We prove matching direct and inverse theorems for (algebraic) polynomial approximation with doubling weights $w$ having finitely many   zeros and singularities (i.e.,  points where $w$ becomes infinite) on an interval and not too ``rapidly changing'' away from these zeros and singularities. This class of doubling weights
 is rather wide and, in particular, includes the classical Jacobi weights, generalized Jacobi weights and generalized Ditzian-Totik weights.
  We approximate in the weighted $\Lp$ (quasi) norm $\norm{f}{p, w}$ with $0<p<\infty$, where   $\norm{f}{p, w} := \left(\int_{-1}^1 |f(u)|^p w(u) du \right)^{1/p}$. Equivalence type results involving related realization functionals are also discussed.
\end{abstract}

%

 \sect{Introduction}

The main goal of this paper is to prove matching direct and inverse theorems for polynomial approximation with doubling weights $w$ having finitely many   zeros and singularities (\ie  points where $w$ becomes infinite) on an interval and not too ``rapidly changing''. In order to discuss this further, we need to
recall some notation and definitions. As usual, $\Lp(I)$, $0<p<\infty$, is the set of measurable on $I$ functions $f$ equipped with the (quasi)norm $\norm{f}{\Lp(I)} := \left(\int_I |f(u)|^p du \right)^{1/p}$.
We say that a   function $w$ is a doubling weight on $[-1,1]$ if $w\in\L_1[-1,1]$ is nonnegative, not identically equal to zero, and   there exists a positive constant $L$ (the so-called doubling constant of $w$) such that
$w(2I) \leq L w(I)$, for any interval $I\subset [-1,1]$. Here, $w(J) := \int_{J \cap [-1,1]} w(u) du$, and $2I$ denotes the interval of length $2|I|$ ($|I|$ is the length of   $I$) with the same center as $I$.
Doubling weights, their properties and various approximation results are discussed  in a series of papers \cites{mt2000, mt2001, mt1998, mt1999} by G. Mastroianni and V. Totik. In particular, it turns out that one can obtain many analogs of theorems for unweighted approximation by considering weights $w_n$ which are certain averages of $w$ depending on the degree of approximating polynomials. Recall (see \eg \cite{mt2001}) that $w_n(x) := \dnx^{-1} \int_{x-\dnx}^{x+\dnx} w(u) du$, where $\dnx = n^{-1}(1-x^2)^{1/2} + n^{-2}$. We refer the reader to \cites{mt2001, mt1998} and \cite{k-weighted} for further discussions of results involving $w_n$.
At the same time, it is clear that averaging removes singularities (and ``lifts'' zeros) of weights, and so a natural question is whether or not one can obtain matching direct and   inverse theorems for general doubling weights. This seems to be a very hard question since a general doubling weight can exhibit some rather ``wild'' behavior that makes it  hard if not impossible to work with (while proving positive approximation results). For example, doubling weights can vanish on sets of positive measure as well as they can be ``rapidly changing''.
Even relatively well-behaved weights (such as generalized Jacobi weights) can cause difficulties because of the presence of internal zeros/singularities. For example, see \cite{mt2001} for discussions of difficulties in forming weighted moduli of smoothness for generalized Jacobi weights, and  \cite{mt1999, dbmr} for   examples showing that the original Jackson-Favard estimates are no longer valid for some specific doubling weights.


Still, if a doubling weight $w$ has only finitely many zeros and singularities inside $[-1,1]$ and is not too rapidly changing once one moves away from these points (\ie if it behaves like $w_n$ there), the matching direct and inverse results are possible (this is the main result of this paper).
Earlier, this type of results was established in \cite[Theorem 1.4]{mt2001} in the uniform norm weighted by   generalized Jacobi weights with finitely many zeros in $[-1,1]$, and in
 \cite[Theorem 3.1]{dbmr} in the $\Lp$ norm (with $1\leq p \leq \infty$) weighted by  a specific weight having one zero at the origin and zeros or singularities at $\pm 1$ (see also \cites{cm,dbmv} for related results).
 However, we are not aware of any results of this type for $0<p<1$. Perhaps, the reason for this is that the usual method seems to be to first establish the equivalence of the moduli and some related $K$-functionals, and then proceed with the proofs. This method cannot work for $0<p<1$ since it  is rather well known (see \cite{dhi}) that $K$-functionals are often zeros if $0<p<1$.

 Our approach is different and is actually somewhat similar to the one used in our earlier paper \cite{k-weighted} where matching direct and inverse theorems were established for the weights $w_n$ and all $0<p\leq \infty$.
 Namely, we derive the equivalence of the moduli and related   ``realization'' functionals as a corollary of our estimates, and our proofs of direct/inverse theorems does not rely on this equivalence.

 The class of doubling weights $\W(\Z)$ that we introduce in Section~\ref{propdouble}
 is rather wide and, in particular, includes the classical Jacobi weights, generalized Jacobi weights and generalized Ditzian-Totik weights.
 We approximate in the weighted $\Lp$ (quasi) norm with $0<p<\infty$. For $p<\infty$, the weighted (quasi)norm is defined as $\norm{f}{\Lp(I), w} := \left(\int_I |f(u)|^p w(u) du \right)^{1/p}$ and $\norm{f}{p, w} := \norm{f}{\Lp[-1,1], w}$. We also denote by $\Lp^w$ the set of all functions on $[-1,1]$ such that $\norm{f}{p, w} < \infty$.

 Many of the results presented in this paper are also valid if $p=\infty$. However, one can only approximate essentially bounded functions by polynomials if the weights are essentially bounded. This puts a rather significant restriction on the weights, and the weights having the so-called property $A^*$ are usually considered if $p=\infty$ instead of a  wider class of doubling weights. This is the main reason why
we only discuss the case $0<p<\infty$ in this paper, and analogous results for $p=\infty$ and $A^*$ weights will appear elsewhere. 

The paper is organized as follows. In Section~\ref{propdouble}, we define a class of doubling weights  $\W(\Z)$  with finitely many zeros and singularities inside $[-1,1]$ and give several equivalent conditions guaranteeing that $w$ is in this class. Main part weighted moduli, (complete) weighted moduli as well as averaged moduli of smoothness are introduced in Section~\ref{moduli}. A relation between the degrees of local approximation by piecewise polynomials and the main part moduli is established in Section~\ref{locapp}. \lem{lem2.1} in this section is the main result that allows us to estimate the degree of approximation away from zeros and singularities of the weight $w$.
A Jackson type theorem with doubling weights  from the class $\W(\Z)$ is proved in Section~\ref{jacksec}. This is the main direct result in this paper. In Section~\ref{remezmb}, we discuss several Remez type and Markov-Bernstein type results that are needed in the proof of the inverse theorems. In Section~\ref{cruciallemmas}, we prove two crucial lemmas on local approximation of polynomials of degree $<n$ by Taylor polynomials of degree $<r$ (lemmas deal  with cases $1\leq p <\infty$ and  $0<p<1$ separately). The inverse results heavily depend on these lemmas. Some preliminary results needed in the proofs of inverse theorems are given in Section~\ref{prelinv}. The inverse theorems in   cases $1\leq p < \infty$ and $0<p<1$ are proved, respectively, in Sections~\ref{invbig} and \ref{invsmall}. In Section~\ref{real}, we obtain realization type results by proving the equivalence of the averaged and ``regular'' weighted moduli and appropriate realization functionals. Finally, an auxiliary result that is well known in the unweighted case  about a polynomial of near best approximation (in weighted $\Lp$ with $w$ from the class $\W(\Z)$) being a near best approximant on a slightly larger interval is proved in Section~\ref{appendix}. This result is needed in the proof of the direct theorem and it could be used to provide an alternative proof of relations between different moduli with different parameters $A$.

 \sect{Doubling weights with finitely many zeros and singularities} \label{propdouble}

 Let $w$ be a doubling weight on $[-1,1]$    such that
   $w(z)=0$ or $w(z)=\infty$ at finitely many points $z$.
Moreover, we assume that $w(x)$ ``does not rapidly change'' when $x$ is ``far'' from these points $z$. These      assumptions certainly limit the set of the weights that we consider since there are  doubling weights that vanish on sets of positive measure and, at the same time, there are ``rapidly changing'' positive doubling weights. However, many important weights (such as generalized Jacobi weights or the so-called generalized Ditzian-Totik weights, for example) satisfy this property (see below for their definitions).

We now make everything precise in the following definition noting that, throughout this paper, if $y<x$, then $[x,y] := [y,x]$ (and not $\emptyset$ as it is sometimes defined).
We also denote $\varphi(x):= (1-x^2)^{1/2}$, $\rho(h,x) := h\varphi(x)+h^2$ and note that $\dnx = \rho(1/n, x)$.

  \begin{definition} \label{def1.1}
  Let $M\in\N$ and $\Z := (z_j)_{j=1}^M$,  $-1 \leq    z_1 < \dots < z_{M-1} < z_M \leq  1$. We say that a doubling weight $w$ belongs to the class
  $\W(\Z)$ (and write $w\in \W(\Z)$) if,
 for any $\e>0$ and  $x, y\in [-1,1]$ such that $|x-y| \leq  \rho(\e, x)$ and
 $\dist\left( [x,y], z_j \right) \geq \rho(\e, z_j)$ for all $1\leq j \leq M$, the following inequalities are satisfied
 \be \label{nochange}
c_* w(y) \leq  w(x) \leq c_*^{-1} w(y) , 
 \ee
 where the   constant $c_*$ depends only on $w$, and does not depend on $x$, $y$ and $\e$.
 \end{definition}

 Note that the set $\Z$ is where $w$ can have zeros or singularities, but
  we do not actually require that it happens at all points in $\Z$. In other words, we do not exclude the possibility that $w$
 is ``well behaved'' at some/all points in $\Z$. We also note that the set $\Z$ is considered fixed throughout this paper, and so we refer to it in various theorems without redefining it (unless a statement/example is given for a specific $\Z$ in which case it will be explicitly stated). Also, note that the moduli of smoothness that we define below depend on $\Z$ and so, in particular, all constants in our estimates involving moduli will depend on $M$, but we are not   explicitly stating this in every statement.

 It is convenient to denote
 \[
 \Z_{A,h}^j  := \left\{ x \in [-1,1] \st |x-z_j| \leq A \rho(h, z_j)  \right\},  \quad   1\leq j\leq M   ,
 \]
 \[
 \Z_{A,h}  := \cup_{j=1}^M \Z_{A,h}^j,
 \]
 and
 \[
 \I_{A, h} :=  \left([-1,1] \setminus  \Z_{A,h}\right)^{cl}   = \left\{ x\in [-1,1] \st |x-z_j| \geq  A \rho(h,z_j), \; \text{for all } 1\leq j \leq M \right\}.
 \]
Also,
 \[
  \D := \D(w) := \pmin  \left\{  |z_j - z_{j-1}| \st 1\leq j\leq M+1 \right\} ,
 \]
 where $z_0 := -1$, $z_{M+1} := 1$
 and
 $\pmin(S)$ is the smallest {\em positive} number from the set $S$ of nonnegative reals.

Note that the condition $\dist\left( [x,y], z_j \right) \geq \rho(\e, z_j)$, $1\leq j \leq M$, in Definition~\ref{def1.1}  is equivalent to $[x,y] \subset \I_{1, \e}$.


Throughout this paper,  $(x_i)_{i=0}^n$ is the Chebyshev partition of $[-1,1]$, \ie
 $x_i = \cos(i\pi/n)$, $0\leq i\leq n$. For $1\leq i\leq n$, we also denote $I_i := [x_i, x_{i-1}]$,
 \[
 \psi_i :=    \psi_i(x) :=  \frac{|I_i|}{|x-x_i|+|I_i|}
\andd
\chi_i(x) := \chi_{[x_i, 1]}(x) =
\begin{cases}
1, & \mbox{\rm if } x_i \leq x \leq 1, \\
0, & \mbox{\rm otherwise}.
\end{cases}
\]

We need the following facts about the Chebyshev partition and the weights $w_n$.
\begin{itemize}
\item $\dnx \leq |I_i| \leq 5\dnx$ for all $x\in I_i$ and $1\leq i \leq n$.
\item $|I_i|/3 \leq |I_{i+1}| \leq 3|I_i|$ for all $1\leq i \leq n-1$.
\item If $\a\geq 2$, then $\sum_{i=1}^n \psi_i(x)^\a \leq c$ for all $-1\leq x\leq 1$, and $\int_{-1}^1 \psi_i(x)^\a dx \leq c |I_i|$ for all $1\leq i \leq n$.
\item For all $x,y\in [-1,1]$, $\dn(y)^2 \leq 4 \dnx (|x-y|+\dnx)$.
\item For any $c_0>0$ and $x\in [-1,1]$, the interval $[x-c_0\dnx, x+c_0\dnx]$ has nonempty intersection with at most $m$ intervals $I_i$, $1\leq i \leq n$, where $m$ is some natural number that depends only on $c_0$. This follows from \prop{auxprop} whose proof we postpone until Section~\ref{locapp}.
\item  For any doubling weight $w$, if $n\sim m$, then $w_n(x) \sim w_m(x)$, for all $x\in [-1,1]$.
 \item For any doubling weight $w$ and $n\in\N$, $w_n(x) \sim w_n(y)$ if $|x-y|\leq c_* \dnx$, with   equivalence constants  depending only on $c_*$ and the doubling constant of $w$ (see  \cite[(2.3)]{mt2001}).
 \item For any doubling weight $w$, $n\in\N$, $1\leq i \leq n$, $x\in [-1,1]$ and $y\in I_i$, $w_n(x) \leq c \psi_i(x)^{-s} w_n(y)$ and $w_n(y) \leq c \psi_i(x)^{-s} w_n(x)$, where constants $c$ and $s \geq 0$ depend only on the doubling constant of $w$ (see \cite[Lemma 2.5]{k-weighted}).
\end{itemize}

We also mention that defining $I_i$'s to be closed   causes some ambiguity at the boundaries of these intervals since any two adjacent intervals in this partition have a nonempty intersection. Hence, when we make statements of type ``let $x\in [-1,1]$ and let $\mu$ be such that $x\in I_\mu$'', this is ambiguous if $x = x_j$ for some $1\leq j \leq n-1$, since there are actually two intervals   containing $x$ (namely, $I_j$ and $I_{j+1}$).
To remedy this problem, we   use the convention that, if $x$ belongs to two adjacent (closed) intervals, we   always choose the right interval as the one containing $x$.

We are now ready to discuss several conditions that are equivalent to the statement that a doubling weight is in the class $\W(\Z)$.

\begin{lemma} \label{properties}
Let $w$ be a doubling weight. The following conditions   are equivalent.
\begin{enumerate}[(i)]
\item \label{i} $w\in \W(\Z)$.
\item \label{ii} For any $n\in\N$ and $x, y$ such that $[x,y] \subset \I_{1, 1/n}$ and $|x-y| \leq  \dnx$, inequalities \ineq{nochange} are satisfied with the constant $c_*$ depending only on $w$.

\item \label{iii}
For some $N\in\N$ that depends only on $w$, and any  $n\geq N$ and $x, y$ such that $[x,y] \subset \I_{1, 1/n}$ and $|x-y| \leq  \dnx$, inequalities \ineq{nochange} are satisfied with the constant $c_*$ depending only on  $w$.

\item  \label{iv} For any $n\in\N$, $A, B >0$, and $x, y$ such that $[x,y] \subset \I_{A, 1/n}$ and $|x-y| \leq B \dnx$, inequalities \ineq{nochange} are satisfied with the constant $c_*$ depending only on
$w$, $A$ and $B$.
\item \label{v}
For  any $n\in\N$ and $A>0$,
 \[
  w(x) \sim  w_n(x) , \quad x \in \I_{A, 1/n} ,
\]
where the equivalence constants depend only on $w$ and $A$, and are independent of $x$ and $n$.


\end{enumerate}
\end{lemma}

\begin{proof} 
 Clearly,  \ineq{i} $\Rightarrow$ \ineq{ii} (one just needs to pick $\e  = 1/n$), and \ineq{ii} $\Rightarrow$ \ineq{iii}.
 Also, \ineq{iv} $\Rightarrow$ \ineq{i}. Indeed, note that the statement of Definition~\ref{def1.1} becomes vacuous if $\e > \sqrt{2}$ (since $\rho(\e, z_j) > 2$). Hence,   assuming that $\e \leq \sqrt{2}$
  we pick
$n = \lfloor 2/\e \rfloor \in \N$, $A=1$ and $B=4$.   Then $1< n\e \leq 2$, $A \dn(z_j) \leq \rho(\e, z_j)$ and $B \dn(x) \geq \rho(\e, x)$, and so if
$[x,y] \subset \I_{1, \e}$ and $|x-y|\leq \rho(\e, x)$, then $[x,y] \subset \I_{A, 1/n}$ and $|x-y| \leq B \dnx$.

Now, we will show that  \ineq{iii} $\Rightarrow$ \ineq{iv}.
 Let $n\in\N$ and $A, B>0$ be given, and suppose that $x,y$ are such that $[x,y] \subset \I_{A, 1/n}$ and $|x-y| \leq B \dnx$.
 Pick $m:= \max\left\{ \lceil n/\min\{A, 1\} \rceil, N \right\}$ and note that $A \dn(z_j) \geq \dm(z_j)$, and so $\I_{A, 1/n} \subset \I_{1, 1/m}$. Also, it is not difficult to check that
  $m/n \leq \max\left\{ N,  2/\min\{A, 1\}\right\} =: c^*$ and hence
 $\dnx \leq (c^*)^2 \dmx $ which implies that $|x-y| \leq B (c^*)^2 \dmx =: M \dmx$.

 Hence, in order to complete the proof it is sufficient to show that, for any $x,y\in [-1,1]$ such that $|x-y| \leq M \dmx$ there are $K$ points $y_i$, $1\leq i\leq K$, with $K\in\N$ depending only on $M$,  such that
 \[
 [x,y] \subset \cup_{i=1}^{K-1} [y_i, y_{i+1}] \andd |y_i - y_{i+1}| \leq \dm(y_i) , \; 1\leq i \leq K-1 .
 \]
We will use \prop{auxprop}. Let $(x_i)_{i=0}^m$ be the Chebyshev partition of $[-1,1]$ into  $m$ intervals $I_i = [x_i, x_{i-1}]$. Suppose that  $x\in  I_\mu$, $1\leq \mu \leq m$,   denote
\[
I^* := \left\{ 1\leq i \leq m \st I_i \cap [x,y] \neq \emptyset \right\} \andd
I_* := \left\{ 1\leq i \leq m \st I_i \subset [x,y]   \right\},
\]
and let $J:= \cup_{i\in I^*} I_i$.

If $I_* = \emptyset$, then 
$[x,y] \subset I_\mu \cup I_{\mu \pm 1}$, and $J$ consists of at most $2$ intervals $I_i$.
If $I_* \neq \emptyset$, then
recalling that $|I_{i\pm 1}| \leq 3|I_i|$
and  $\dmx \leq |I_i| \leq 5 \dmx$, for any $x\in I_i$,
we conclude
 \[
 |J| \leq    7 \left| \cup_{i\in I_*} I_i \right| \leq 7 |x-y| \leq 7M \dmx \leq 7M |I_\nu| .
\]
\prop{auxprop} implies that $J$ consists of at most $k$ intervals $I_i$, where $k$ depends only on $M$.
We now define $y_j^i := x_j + i|I_j|/5$, $0\leq i\leq 5$, for all $j\in I^*$, and denote
$Y:= (y_i)_{i=1}^K := \cup_{j\in I^*} \{y_j^0, y_j^1, \dots, y_j^5 \}$, where $y_i < y_{i+1}$, $1\leq i \leq K-1$.
Then, $K$ is not bigger than $5k+1$ and depends only on $M$, $[x,y] \subset J = \cup_{i=1}^{K-1} [y_i, y_{i+1}]$, and, for each $1\leq i \leq K-1$,
$|y_i - y_{i+1}| \leq |I_{j(i)}|/5 \leq \dm(y_i)$.

So far, we have verified the equivalence of \ineq{i}-\ineq{iv}.

We will show now that \ineq{iv} $\Rightarrow$ \ineq{v}.

Let $A>0$ and
 suppose that $n\in\N$
is such that
$n > 4(A+1)/\D$.
 This guarantees that
\[
2\dnx < \D -A \left[\dn(z_j)+ \dn(z_{j-1})\right]      \leq z_j - A\dn(z_j) - \left( z_{j-1} + A\dn(z_{j-1}) \right), \quad 1\leq j\leq M+1, \; z_{j-1}\neq z_j ,
\]
and so
 $[x-\dnx, x+\dnx]$ has a nonempty intersection with at most one interval from $\Z_{A, 1/n}$. Moreover,   if $[x-\dnx, x+\dnx]$ does intersect an interval from $\Z_{A, 1/n}$, then it does not contain $\pm 1$.
Hence, if  $x\in \I_{A, 1/n}$, then either
$[x, x+\dnx] \subset \I_{A, 1/n}$ or $[x-\dnx, x]\subset \I_{A, 1/n}$,
and without loss of generality, suppose that it's the former.
Then, taking into account that
\[
w\left( [x-\mu, x+\mu]\right) \leq w\left( [x-\mu, x+2\mu]\right) 
\leq L^2 w\left( [x+\mu/8, x+7\mu/8]\right) \leq L^2 w\left( [x , x+\mu]\right) ,
\]
we have
 \[
 w_n(x) = {1 \over \dnx} \int_{x-\dnx}^{x+\dnx} w(u) du \leq {L^2 \over \dnx} \int_{x}^{x+\dnx} w(u) du  \leq  L^2 c_*^{-1} w(x)
 \]
and
 \[
 w_n(x) \geq {1 \over \dnx} \int_{x}^{x+\dnx} w(u) du \geq  c_*  w(x) ,
\]
where $c_*$ depends only on $w$ and $A$.

Hence, \ineq{v} is proved for all $n\in\N$ such that $n > 4(A+1)/\D$.
If $1\leq n \leq N:= \lceil 4(A+1)/\D \rceil$, then we use the fact that \ineq{v} is valid for $n = N+1$, 
$\I_{A, 1/n} \subset  \I_{A, 1/(N+1)}$ and
  that
  $w_n(x) \sim w_{N+1}(x)$ with equivalence constants depending only on $w$ and $N$, to conclude that
  \[
  w(x) \sim w_{N+1}(x) \sim w_n(x) , \quad x \in \I_{A, 1/n} .
  \]

To prove \ineq{v} $\Rightarrow$ \ineq{iv}, we note that
it follows from the doubling condition that, if $x,y \in [-1,1]$ and  $|x-y|\leq B\dnx$, then $w_n(x) \sim w_n(y)$
with equivalence constants   depending only on $B$ and the doubling constant of $w$. Hence, if \ineq{v} is valid and $x,y$ are such that $[x,y] \subset \I_{A, 1/n}$ and $|x-y| \leq B \dnx$, then
\[
w(y) \sim w_n(y) \sim w_n(x) \sim w(x)
\]
with equivalence constants depending on $A$, $B$ and the weight $w$. This verifies \ineq{iv}.
\end{proof}

 \begin{remark}
 We note that if a doubling weight $w$ is in the class $\W(\Z)$ then, in particular,  it is bounded away from zero and $\infty$ when $x$ is ``far'' from $\Z$.
 In other words,
 \[
 \forall \e>0   \; \exists \delta_\e>0 : \delta_\e <  w(x) < \delta_\e^{-1}, \quad \text{for all $x$ such that }\; \dist(x, \Z) \geq \e .
 \]
 This follows from \lemp{iv} if we pick $n = \lceil 2/\e \rceil$, $A=1$ and $B = 2n^2$.
\end{remark}

We will now  show  that  if a doubling weight $w$ is monotone near points from $\Z$ and is bounded away from zero and infinity on the rest of the interval $[-1,1]$ then it is in the class $\W(\Z)$.

We use the usual notation $f_+(a) := \lim_{x \to a^+} f(x)$ and $f_-(a) := \lim_{x \to a^-} f(x)$.


 \begin{lemma} \label{lem1.6}
Let $w$ be a doubling weight, and suppose that there exists $0< \a < \D/4$ such that $w$ is monotone on $(z_j - \a, z_j) \cap[-1,1]$ and on $(z_j, z_j+\a)\cap[-1,1]$ for all $1\leq j\leq M$, and suppose that
$\mu_*>0$ and $\mu^* < \infty$, where
\[
\mu_*  := \min \left\{ \inf_{x\in S_\a} w(x)  ,   \min_{1\leq j \leq M} \left\{ w_-( z_j+\a ),  w_+(z_j-\a)  \right\} \right\}
\]
and
\[
\mu^*  := \max \left\{ \sup_{x\in S_\a} w(x)  ,   \max_{1\leq j \leq M} \left\{ w_-( z_j+\a),  w_+(z_j-\a)  \right\} \right\}
\]
where $S_\a := \left\{ x \in [-1,1] \st \dist(x, \Z) \geq \a \right\}$.

  Then $w$ belongs to the class $\W(\Z)$.
\end{lemma}

We use the convention that if a quantity is not defined then it is not present in the set whose minimum or maximum is   taken. Thus, for example, if $z_1 = -1$, then
$w_-(-1-\a)$ is excluded from the definition of $\mu_*$ and $\mu^*$ in the statement of the lemma since this quantity is not defined.

\begin{proof}
For each $1\leq i\leq M$, there exists $\e_i >0$ such that
\[
\mu_*/2 \leq  w(x) \leq 2 \mu^* , \quad \mbox{\rm for all }\;   x\in \left([z_i + \a - \e_i, z_i + \a] \cup [z_i - \a,   z_i -\a + \e_i]  \right) \cap [-1,1] .
\]
We let
$\e := \min\left\{ \a/2, \min_{1\leq j \leq M} \e_i \right\}$ and $N :=  \lceil 4/\e \rceil$.
Note that   $N$ depends only on the weight $w$, and that the inequality $\dnx \leq \e/2$  is satisfied for all $x\in [-1,1]$ and   all $n\geq N$.
Recalling that
\[
S_{\a-\e} = \left\{ x \in [-1,1] \st \dist(x, \Z) \geq \a-\e \right\}
\]
 we also note that $\mu_*/2 \leq  w(x) \leq 2 \mu^*$, for all $x\in S_{\a-\e}$.

Now, let $n\geq N$ and let $x,y$ be such that $[x,y] \subset \I_{1, 1/n}$ and $|x-y| \leq \dnx$. We will show that \lemp{iii} is valid which implies that $w$ is in the class $\W(\Z)$.
We have the following cases to consider (for convenience, suppose that $x<y$):
\begin{enumerate}[(a)]
\item \label{a} $[x,y] \subset S_{\a-\e}$,
\item \label{b} $[x,y] \cap \left( [-1,1] \setminus S_{\a-\e} \right) \neq \emptyset$.
\end{enumerate}

\medskip\noindent
{\bf Case \ineq{a}}:  $ \mu_*/2 \leq w(x), w(y) \leq 2 \mu^*$, and so \ineq{nochange} is satisfied with $c_* = \mu_*/(4\mu^*)$.

\medskip\noindent
{\bf Case \ineq{b}}:
Let $I_x := [x-\dnx/6, x]$ and $I_y := [y, y+\dnx/6]$ and note that
$I_{x,y} := [x-\dnx/6,  y+\dnx/6]$ is such that  $I_{x,y} \cap \{\z_i \pm \a\}_{i=1}^M = \emptyset$ since
\[
\dist(I_{x,y}, \{\z_i \pm \a\}_{i=1}^M) \geq \dist([x,y], \{\z_i \pm \a\}_{i=1}^M) - \dnx/6 \geq    \e/2 - \dnx/6  \geq \e/2 - \e/12 > 0.
\]

Additionally,
 $I_{x,y} \cap \Z = \emptyset$.
Indeed, recalling that $\dn(v)^2 \leq 4 \dn(u) \left( |v-u|+ \dn(u)\right)$, for all $u, v\in [-1,1]$,   letting $u=z_j$, $1\leq j\leq M$, and $v=x$ and noting that
     $|x-z_j| \geq \dn(z_j)$ because $x\in \I_{1,1/n}$, we have
\[
\dnx^2 \leq 4 \dn(z_j) \left( |x-z_j|+ \dn(z_j)\right) \leq 8 |x-z_j|^2 ,
\]
and so $|x-z_j| > \dnx/3$ for all $1\leq j\leq M$.

Also, taking into account that $y\in \I_{1,1/n}$ which implies $|y-z_j| \geq \dn(z_j)$ we have
\[
\dnx^2 
\leq  4 |y-z_j| \left( |x-z_j|+ |y-z_j|\right) \leq 4 |y-z_j| \left( |x-y|+ 2|y-z_j|\right)
\leq 4 |y-z_j| \left( \dnx + 2|y-z_j|\right),
\]
which implies that $\dnx < 6 |y-z_j|$ for all $1\leq j\leq M$.

Therefore, $I_{x,y} \cap  \{ z_i,  z_i \pm \a\}_{i=1}^M = \emptyset$, and so $w$ is monotone on $I_{x,y}$.

It  follows from the properties of doubling weights (see \cite[Lemma 2.1]{mt2000}, for example) that
$c_0 w(I_x) \leq w(I_y) \leq c_0^{-1} w(I_x)$ (since $|I_x| = |I_y| \sim |I_{x,y}|$) with the constant $c_0$ depending only on $w$.

Now, if  $w$ is nondecreasing on $I_{x,y}$, then
\[
w(x) \leq w(y) \leq 6 w(I_y)/\dnx \leq 6 c_0^{-1} w(I_x)/\dnx \leq  c_0^{-1} w(x) ,
\]
and if $w$ is nonincreasing on $I_{x,y}$, then
\[
w(y) \leq w(x)\leq 6 w(I_x)/\dnx \leq 6 c_0^{-1} w(I_y)/\dnx \leq  c_0^{-1} w(y) .
\]
This verifies \lemp{iii}, and the proof is now complete.

\end{proof}

\begin{corollary} \label{cor1.5}
Let $w$ be a doubling weight, and suppose that $w$ is piecewise monotone with finitely many monotonicity intervals, \ie let $T:=(t_i)_{i=0}^K$, $K\in \N$, be such that
$-1=t_0 <t_1 < \dots < t_{K-1} < t_K = 1$ and $w$ is monotone on each interval $(t_i, t_{i+1})$, $0\leq i \leq K-1$. Moreover, assume that
$\mu^* <\infty$ and $\mu_* > 0$, where
\[
\mu^* := \max\left\{ w(t_i), w_{\pm}(t_i) \st 0\leq i \leq K , \; t_i \not\in \Z\right\} \andd \mu_* := \min\left\{ w(t_i), w_{\pm}(t_i) \st 0\leq i \leq K , \; t_i \not\in \Z\right\}
\]
(with the convention that $\max\{\emptyset\} = \min\{\emptyset\} :=1$,  $w_{-}(-1):= w(-1)$ and $w_+(1):= w(1)$).   Then $w$ belongs to the class $\W(\Z)$.
\end{corollary}

Taking into account characterization of monotone doubling weights (see \eg \cite{cu}) and \lem{lem1.6}, it is now relatively easy to check that many well known weights are not only doubling but  are also in $\W(\Z)$ for some $\Z$.

\begin{example} The following are examples of doubling weights from $\W(\Z)$ with $\Z = (z_i)_{i=1}^M$, $-1   \leq   z_1 < \dots < z_{M-1} < z_M \leq  1$.
\begin{itemize}
\item  Classical Jacobi weights:
\[
w(x) = (1+x)^\a (1-x)^\b , \quad \a,\b > -1,  \quad \mbox{\rm with  $M=2$, $z_1 = -1$ and  $z_2= 1$.}
\]

\item
Generalized Jacobi weights:
\[
w(x) = \prod_{j=1}^M |x-z_j|^{\gamma_j}, \quad \gamma_j > -1  .
\]

\item   Generalized DT weights  (see \eg in \cite[p. 134]{cm}):
\[
w(x) = \prod_{j=1}^M |x-z_j|^{\gamma_j} \left(\ln {e \over |x-z_j|} \right)^{\Gamma_j}      , \quad \gamma_j >-1, \; \Gamma_j \in\R .
\]
(Note that if these weights are   defined with $\gamma_j = -1$, $\Gamma_j <-1$, for some $j$'s,  then they will be in $\L_1$ but will not be doubling. For example, $w(x) = |x|^{-1} (1-\ln|x|)^{-2}$ is not doubling since, for example, for sufficiently small $t>0$,
$w([0,t]) \sim (1-\ln t)^{-1}$ and $w([t,2t]) \sim (1-\ln t)^{-1} (1-\ln(2t))^{-1}$ and so $w([0,t])/w([0, 2t]) \to \infty$ as $t \to 0^+$, which cannot happen for doubling weights.)
 \end{itemize}
\end{example}

\begin{remark}
Of course, there are doubling weights which are not in any $\W(\Z)$ classes. Doubling weights that vanish on a set of positive measure (see \cite[Chapter I, Section 8.8]{stein} for an example) is an illustration of this.
Also, there are doubling weights which are not $A_\infty$ weights and which do not vanish anywhere (see \cites{fm, mt2000}), and one can use the same construction for any $\Z$ to build a doubling weight $w$ which will not be in $\W(\Z)$.
\end{remark}

\sect{Moduli of smoothness}\label{moduli}

As usual, for $r\in\N$, let
\[\Delta_h^r(f,x, S):=\left\{
\begin{array}{ll} \ds
\sum_{i=0}^r  {r \choose i}
(-1)^{r-i} f(x-rh/2+ih),&\mbox{\rm if }\, [x-rh/2, x+rh/2]  \subset S \,,\\
0,&\mbox{\rm otherwise},
\end{array}\right.\]
be  the $r$th symmetric difference. Note that $S$ can be a union of (disjoint) intervals. Also, let    $\Delta_h^r(f,x) := \Delta_h^r(f,x, [-1,1])$.

{\em Main part weighted modulus of smoothness} is defined as
 \begin{eqnarray*}
 \Omega_\varphi^r(f, A, t)_{p, w} &:= & \sup_{0<h\leq t} \norm{\Delta_{h\varphi(\cdot)}^r(f, \cdot, \I_{A, h})}{\Lp(\I_{A, h}),w} .
\end{eqnarray*}

Note that, for small $A$ and $h$,  $\I_{A, h}$ consists of $M-1$, $M$ or $M+1$ intervals depending on whether or not $w$ has a zero/singularity at $\pm 1$.


It is clear  that moduli $\Omega_\varphi^r$ are not sufficient to characterize smoothness of functions (the main part weighted modulus is obviously zero for any piecewise constant function $f$ with jump points at $\Z$), and we define the (complete) {\em weighted modulus of smoothness } as
 \begin{eqnarray} \label{complete}
 \w_\varphi^r(f, A,   t)_{p, w}  :=  \Omega_\varphi^r(f, A, t)_{p, w}  + \sum_{j=1}^M  E_r(f)_{\Lp(\Z_{2A,t}^j), w}  , 
\end{eqnarray}
where
\[
E_r(f)_{\Lp(I), w} :=    \inf_{q  \in\Poly_r} \norm{f-q }{\Lp(I), w}
\]
(see \eg \cite[Chapter 11]{dt} and \cites{mt2001, dbmr}  for similar definitions).
Note that these moduli can be defined as $\w_\varphi^r(f, A, B,  t)_{p, w}$ with $2A$ in the sets $\Z_{2A,t}^j$ replaced by $B$.
It is possible to show that $\w_\varphi^r(f, A, B,  t)_{p, w}$ are equivalent for different $A$ and $B$ provided $B>A$ and $t$ is small (if $0<p<1$), and we
did not investigate the question of equivalence of these moduli in the  case $B\leq A$.
 It will be shown in Section~\ref{real} that moduli \ineq{complete} (as well as the averaged moduli \ineq{aver} defined below)  are equivalent for all positive $A$ and all $t>0$ (if $1\leq p <\infty$) or $0<t<t_0$, for some $t_0>0$ (if $0<p<1$). Note, however, that we cannot use this equivalence in the proof of the direct theorem (which would simplify it considerably) since we derive it as a corollary of several results, the direct theorem being one of them.

We   define  the  {\em averaged main part weighted modulus} and the (complete) {\em  averaged weighted modulus of smoothness}, respectively, as
 \begin{eqnarray*}
\widetilde\Omega_\varphi^r(f, A, t)_{p, w} &:=&  \left(\frac{1}{t}  \int_0^{t}  \int_{\I_{A,h}}       w (x) |\Delta_{h\varphi(x)}^r(f, x, \I_{A,h})|^p dx dh \right)^{1/p}\\
&=&  \left(\frac{1}{t}  \int_0^{t}  \norm{ \Delta_{h\varphi(\cdot)}^r(f, \cdot, \I_{A, h})}{\Lp(\I_{A, h}), w}^p  dh \right)^{1/p}
\end{eqnarray*}
and
 \begin{eqnarray} \label{aver}
 \widetilde \w_\varphi^r(f, A,   t)_{p, w}  :=  \widetilde\Omega_\varphi^r(f, A, t)_{p, w}  + \sum_{j=1}^M  E_r(f)_{\Lp(\Z_{2A,t}^j), w} . 
\end{eqnarray}

The following properties of these moduli immediately follow from the definition:
\begin{enumerate}[(i)]
\item   $\ds \widetilde\Omega_\varphi^r(f, A, t)_{p, w} \leq  \Omega_\varphi^r(f, A, t)_{p, w}$ and $\widetilde \w_\varphi^r(f, A,  t)_{p, w} \leq  \w_\varphi^r(f, A, t)_{p, w}$,
 \item
$\ds  \Omega_\varphi^r(f, A, t_2)_{p, w} \leq  \Omega_\varphi^r(f, A, t_1)_{p, w}$ and    $\ds \w_\varphi^r(f, A,   t_2)_{p, w} \leq \w_\varphi^r(f, A,   t_1)_{p, w}$            if $t_1 \geq t_2$,
\item
$\ds  \widetilde \Omega_\varphi^r(f, A,  t_2)_{p, w}  \leq (t_1/t_2)^{1/p} \, \widetilde \Omega_\varphi^r(f, A,   t_1)_{p, w}$  and
$\ds  \widetilde \w_\varphi^r(f, A,   t_2)_{p, w}  \leq (t_1/t_2)^{1/p} \, \widetilde \w_\varphi^r(f, A,   t_1)_{p, w}$
if $t_1 \geq t_2$,
\item
$\ds \Omega_\varphi^r(f, A_1, t)_{p, w} \leq  \Omega_\varphi^r(f, A_2, t)_{p, w}$ and
$\ds \widetilde \Omega_\varphi^r(f, A_1, t)_{p, w} \leq  \widetilde\Omega_\varphi^r(f, A_2, t)_{p, w}$
if $A_1 \geq A_2$
(since $\I_{A_1, h} \subset \I_{A_2, h}$).
\end{enumerate}

We will also need the following auxiliary quantity (``restricted averaged main part modulus'' would be a proper name for it) which will be quite helpful in our estimates:
\[
\widetilde\Omega_\varphi^r(f,  t)_{\Lp(S), w}  :=   \left(\frac{1}{t}  \int_0^{t}  \int_{S}   w (x) |\Delta_{h\varphi(x)}^r(f, x, S)|^p dx dh \right)^{1/p} ,
\]
where $S$ is some subset (a union of intervals) of $[-1,1]$ (that does not depend on $h$).


Note that
\[
\widetilde\Omega_\varphi^r(f,  t)_{\Lp(\I_{A,t}), w}  \leq \widetilde\Omega_\varphi^r(f, A, t)_{p, w} .
\]

We also remark that since $\I_{A, h}$ consists of a number of disjoint intervals when $h$ is small, it is possible to define a main part modulus taking supremum on each of these intervals. In other words, one can define
\[
\Omega_\varphi^{*r}(f, A, t)_{p, w}  :=   \sum_{j=0}^{M} \sup_{0<h\leq t} \norm{\Delta_{h\varphi}^r(f)}{\Lp(J_{A,h}^j) ,w} ,
\]
where $z_0:=-1$, $z_{M+1}:=1$, and $J_{A,h}^j$'s denote components of $\I_{A, h}$, \ie
\[
J_{A,h}^j :=
\begin{cases}
\left[z_{j}+A \rho(h,z_{j}), z_{j+1}-A \rho(h,z_{j+1})\right] , & \mbox{\rm if $1\leq j \leq M-1$},\\
\left[-1, z_1-A \rho(h,z_j)\right] , & \mbox{\rm if $j=0$ and $z_1 \neq -1$},\\
\left[z_M+A \rho(h,z_j), 1\right] , & \mbox{\rm if $j=M$ and $z_M \neq 1$}.
\end{cases}
\]
It is obvious that $\Omega_\varphi^{r}(f, A, t)_{p, w} \leq \Omega_\varphi^{*r}(f, A, t)_{p, w}$, and it is less obvious that this inequality can be reversed for any $f\in \Lp^w$, $0<p<\infty$. Hence, we note that
 $\Omega_\varphi^{*r}$ could replace $\Omega_\varphi^{r}$  everywhere in the proofs below, and so using Corollaries~\ref{corr1} and \ref{corr2} we could actually show that these moduli are equivalent (in the case $0<p<1$,  $t$ would have to be small). However, we are not discussing this further. 

%
%
%
%
%
%

 \sect{Degree of local approximation} \label{locapp}

\begin{proposition} \label{auxprop}
Let $n\in\N$ and suppose that,  for some $1\leq \mu\leq n$,   $I_\mu \subset J$, where $J\subset [-1,1]$ is an interval such that    $|J | \leq c_0 |I_\mu|$. Then there exists  $m\in\N$ depending only on $c_0$ (and independent of $n$) such that $J$   has a nonempty intersection with at most $m$ intervals $I_i$, $1\leq i\leq n$.
\end{proposition}

\begin{proof} If $n=1$, the statement is obvious, and so we assume that $n\geq 2$. Because of symmetry, we may assume that $1\leq \mu \leq \lceil n/2 \rceil$. Now let $1\leq i\leq n$, and compare the distance from $x_i$ to $x_\mu$ to the length of the interval $I_\mu$.
Using the estimates $x/10 \leq \sin x \leq x $, $0\leq x \leq 7\pi/8$,    we have
\begin{eqnarray*}
{|x_i - x_\mu| \over |I_\mu|}  & = &  { \sin\left[ (i+\mu)\pi/(2n)\right]  \sin\left[ |i-\mu|\pi/(2n)\right]  \over   \sin\left[ (2\mu-1)\pi/(2n)\right]  \sin\left[  \pi/(2n)\right] }    \geq {|i^2 - \mu^2| \over 100(2\mu-1)}
  \geq {|i - \mu| \over 200} .
\end{eqnarray*}
If $x_i \in J$, then $|x_i - x_\mu| \leq |J| \leq c_0 |I_\mu|$ and so $|i-\mu| \leq 200 c_0$. This implies that $J$ has empty intersection with all intervals $I_i$ such that
$\min\{ |i-\mu|, |i-1-\mu| \} > 200 c_0$, and so the number of intervals $I_i$ having nonempty intersections with $J$ is $m \leq 400 c_0+2$.
\end{proof}

Recall now that $\w_r(f, t, I)_p := \sup_{0<h\leq t} \norm{\Delta_h^r(f, x, I)}{\Lp(I)}$ is the usual $r$th modulus of smoothness on an interval $I$, and that the well-known Whitney's theorem (see \eg \cite[Theorem 7.1, p. 195]{pp}) implies that, for any $f\in\Lp[a,b]$, $0<p<\infty$,
\[
\inf_{q\in\Poly_r} \norm{f-q}{\Lp[a,b]} \leq c \w_r(f, b-a, [a,b])_p .
\]

\begin{lemma} \label{lem2.1}
Let $w$ be a doubling weight from the class $\W(\Z)$, $0<p < \infty$,  $f\in\Lp^w$,   $n,r\in\N$, and let $A >0$ and $\theta>0$ be arbitrary.
Also, let
\[
I^* := \left\{ 1\leq i \leq n \st I_i \cap \Z_{A, 1/n}   = \emptyset \right\} ,
\]
and suppose that, for each $i\in I^*$, the interval $J_i$ is such that $I_i\subset J_i \subset \I_{A,1/n}$ and $|J_i| \leq c_0 |I_i|$.
Then
 \[
 \sum_{i\in I^*}  w(x_i)\w_r(f, |J_i|, J_i)_p^p \leq c \widetilde \Omega_\varphi^r (f,   \theta/ n)_{\Lp(\I_{A,1/n}), w}^p       ,
 \]
where the   constant $c$   depends only on $r$, $p$, $c_0$, $\theta$, $A$ and the weight $w$.
 \end{lemma}


\begin{proof} The proof is rather standard (see \cite{dly} or \cite[Lemma 5.1]{k-weighted}).
In fact, it is possible to derive an analog of this lemma as a corollary of \cite[Lemma 5.1]{k-weighted} by replacing $f$ by a function $g$ which is identically zero near the points from $\Z$.
However, this approach is not shorter, and we do not immediately get exactly what we need. Hence,
we opted for a direct proof even though it is quite similar to that of \cite[Lemma 5.1]{k-weighted}. We adduce it here for completeness.

The main idea of the proof is the employment of the inequality (see \cite[Lemma 7.2, p. 191]{pp})
\be \label{pp-ineq}
\w_r(f, t, [a,b])_p^p \leq {c \over t} \int_0^t \int_a^b |\Delta_h^r(f, x, [a,b])|^p dx\, dh , \quad 0<p<\infty .
\ee


\prop{auxprop} implies that each $J_i$ has a nonempty intersection with at most $m$ intervals $I_j$, $1\leq j\leq n$, where $m$ depends only on $c_0$.
Since $|I_i| \sim |I_{i\pm 1}| \sim \dn(x_i)$, this implies that
 $\dnx \sim \dn(y)\sim |I_i|$ for all $x,y \in J_i$, and so $|x-y| \leq c \dnx$, for all $x,y\in J_i$.
Hence, since $J_i \subset \I_{A,1/n} $, \lemp{iv} implies that $w(x) \sim w(x_i)$, for all $x\in J_i$, where the equivalence constants depend only on $w$, $A$ and $c_0$.

Taking this into account and  using \ineq{pp-ineq}  we have, for each $i\in I^*$,
\begin{eqnarray*}
w (x_i) \w_r(f, |J_i|, J_i)_p^p &\leq& c w (x_i) \w_r(f, c^*|I_i| , J_i)_p^p \\
& \leq &
  c |I_i|^{-1} \int_0^{c^*|I_i| } \int_{J_i} w (x_i) |\Delta_h^r(f, x, J_i)|^p dx\, dh \\
&\leq&  c  \int_{J_i}  \int_0^{c^*|I_i|/\varphi(x)}  {\varphi(x) \over |I_i|}  w (x) |\Delta_{h\varphi(x)}^r(f, x, J_i)|^p dh \, dx ,
\end{eqnarray*}
where $0<c^* < 1$ is a constant that we will choose later.

Now,    $|I_i| \sim \dn(x) \sim \varphi(x)/n$ for $x\in J_i$, $i \in J^*$, where
\[
J^* := \left\{ i \in I^* \st  J_i \cap (I_1 \cup I_n) = \emptyset \right\} .
\]
Note that depending on whether or not $z_1=-1$ and $z_M=1$ the set $J^*$ may or may not be the same as $I^*$.

Now, for $i \in J^*$, taking into account that $c^* \leq \sqrt{c^*}$,  we have
 \be \label{newaux}
w (x_i) \w_r(f, |J_i|, J_i)_p^p \leq c n  \int_{J_i}  \int_0^{c \sqrt{c^*}/n}    w (x) |\Delta_{h\varphi(x)}^r(f, x, J_i)|^p dh dx .
 \ee
Suppose now that $i \in I^*\setminus J^*$ (we have already remarked that this set may be empty depending on $w$). Recall that    $\Delta_h^r(f, x, J_i)$ is defined to be $0$ if $x\pm rh/2 \not\in J_i$ and, in particular,
$ \Delta_{h\varphi(x)}^r(f, x, J_i) = 0 $ if  $1-|x| < rh\varphi(x)/2$. Therefore, since $\varphi(x)/|I_i| \leq c n \dnx/|I_i| \leq cn$, $x\in J_i$,  for each fixed $x\in J_i$, we have
\[
\int_0^{c^*|I_i|/\varphi(x)}  {\varphi(x) \over |I_i|}  w (x) |\Delta_{h\varphi(x)}^r(f, x, J_i)|^p dh \leq
cn \int_S w (x) |\Delta_{h\varphi(x)}^r(f, x, J_i)|^p dh ,
\]
where
\begin{eqnarray*}
S &:=& \left\{ h \st 0< h \leq \min\left\{ {c^*|I_i| \over  \varphi(x)},  {2(1-|x|) \over r \varphi(x)}  \right\} \right\}\\
& \subset&
\left\{ h \st 0< h \leq c \min\left\{ {c^* \over  n^2  \sqrt{ 1-|x|} },  \sqrt{ 1-|x|}   \right\} \right\}
\subset
\left\{ h \st 0< h \leq c \sqrt{c^*}/n \right\} .
\end{eqnarray*}
Therefore, \ineq{newaux} is valid for $i \in I^*\setminus J^*$ as well. We now choose $c^*$ to be such that $c \sqrt{c^*}$ in the upper limit of the inner integral in \ineq{newaux} is less than $\theta$.
Since each $x$ belongs to finitely many $J_i$'s by \prop{auxprop}, we have
\begin{eqnarray*}
 \sum_{i\in I^*}  w (x_i)\w_r(f, |J_i|, J_i)_p^p  & \leq &
c n  \sum_{i\in I^*}   \int_{J_i}  \int_0^{\theta/n}    w (x) |\Delta_{h\varphi(x)}^r(f, x, J_i)|^p dh dx \\
& \leq &
c n   \int_0^{\theta/n}  \int_{\I_{A,1/n}}      w (x) |\Delta_{h\varphi(x)}^r(f, x, \I_{A,1/n})|^p  dx dh \\
& \leq &
c \widetilde \Omega_\varphi^r (f, \theta/n)_{\Lp(\I_{A,1/n}), w }^p,
\end{eqnarray*}
and the proof is complete.
\end{proof}

 \sect{Jackson type estimate} \label{jacksec}

 The following lemma follows from  \cite[Lemma 3.1]{k-weighted}.

 \begin{lemma} \label{lem5.1}
Let $1\leq i \leq n$, and let $\nu_0, \mu  \in\N_0$ be such that $\mu \geq c_* \max\{\nu_0, 1\}$, where $c_*$ is some sufficiently large absolute (positive) constant. Then the polynomial
$T_i = T_{i} (n, \mu )$ of degree $\leq c(\mu) n$ satisfies the following inequalities for all $x\in [-1,1]$:
\[
\left| T_{i} (x) -   \chi_i(x) \right| \leq c   \psi_i(x)^\mu
\]
and
\[
\left| T_{i}^{(\nu)} (x)  \right| \leq c  |I_i|^{-\nu}   \psi_i(x)^\mu , \quad 0\leq \nu \leq \nu_0 ,
\]
where constants $c$ depend only on $\mu$.
\end{lemma}

We are now ready to state and prove our main direct result.

\begin{theorem} \label{jacksonthm}
Let $w$ be a doubling weight from the class $\W(\Z)$, $r, \nu_0\in\N$, $\nu_0\geq r$, $0<p<\infty$,   and $f\in\Lp^w$.
 Then, there exists $N\in\N$ depending on $r$, $\nu_0$, $p$ and  the weight $w$, such that
 for every $n \geq N$,  $\ccc >0$ and $A>0$,  there exists a polynomial $P_n \in\Poly_n$ such that
\[
\norm{f-P_n}{p, w }   \leq c \widetilde \w_\varphi^r(f, A, \vartheta/n)_{p, w} \leq c   \w_\varphi^r(f, A, \vartheta/n)_{p, w}
\]
and
\[
\norm{ \dn^\nu P_n^{(\nu)}}{p, w } \leq c \widetilde \w_\varphi^r(f, A, \vartheta/n)_{p, w}  \leq c  \w_\varphi^r(f, A, \vartheta/n)_{p, w}   ,      \quad r\leq \nu \leq \nu_0,
\]
where constants $c$  depend only on  $r$, $\nu_0$, $p$, $\ccc$, $A$ and the weight $w$.
\end{theorem}

\begin{proof} The idea of this proof is similar to that of \cite[Theorem 5.3]{k-weighted} where a Jackson type theorem was proved for the weights $w_n$ with moduli of smoothness defined like $\Omega_\varphi^r$ but with $[-1,1]$ instead of $\I_{A, h}$. However, there are some  difficulties that we need to overcome now in order to get the right estimates near $\Z$.

Let $A>0$ and $\ccc>0$ be given (without loss of generality, we can assume that $0<\ccc\leq 1$), and  let $n\in\N$ be sufficiently large (so that each (nonempty) interval $[z_j, z_{j+1}]$, $0\leq j\leq M$,   contains at least $10$ intervals $I_i$), and
let $(x_i)_{i=0}^n$ be the Chebyshev partition of $[-1,1]$. Recall that  $I_i := [x_i, x_{i-1}]$, $1\leq i\leq n$.

For each $1\leq j\leq M$, let
\[
\nu_j := i \quad \mbox{\rm such that}\quad z_j \in I_i
\]
(recall that, if $z_j = x_i$, $1\leq i\leq n$, then we pick the right interval containing $z_j$, \ie $\nu_j = i$ in this case).

Now, we modify   partition $(x_i)_{i=0}^n$ by replacing, for each $1\leq j \leq M$, the knots $x_{\nu_j}$ and $x_{\nu_j-1}$ by $z_j - \newconst_j\dn(z_j)$ and $z+\newconst_j\dn(z_j)$, respectively (replacing only one of them if $z_j$ is $1$ or $-1$). More precisely, for some collection of $M$ constants $0<\newconst_j \leq 1/10$, $1\leq j \leq M$,  which we will choose later, define
\[
\widetilde x_1 := 1-\newconst_M/n^2, \quad \mbox{\rm if $i = 1$ and $z_M=1$},
\]
and
\[
\widetilde x_{n-1} := -1+\newconst_1/n^2, \quad \mbox{\rm if  $z_1=-1$}.
\]
Now, for all   $1\leq i \leq n-1$  where $\widetilde x_i$ has not been defined yet, we let
\[
\widetilde x_i :=
\begin{cases}
z_j - \newconst_j \dn(z_j),   & \mbox{\rm if $i=\nu_j$, $1\leq j \leq M$}, \\
z_j + \newconst_j \dn(z_j),   & \mbox{\rm if $i=\nu_j-1$, $1\leq j \leq M$}, \\
x_i, & \mbox{\rm otherwise}.
\end{cases}
\]
We now note that this new partition $(\widetilde x_i)_{i=0}^n$ has the same properties as the original Chebyshev partition (with constants than now depend on $\newconst_j$). In particular, if $\widetilde I_i := [\widetilde x_i, \widetilde x_{i-1}]$, then
 $|I_i| \sim |\widetilde I_i|$, $|\widetilde I_{i\pm 1}| \sim |\widetilde I_i|$,   $\widetilde \psi_i(x) := |\widetilde I_i| /\left( |x-\widetilde x_i| + |\widetilde I_i| \right) \sim \psi_i(x)$ and
 $|\chi_{[\widetilde x_i,1]}(x) - \chi_{[x_i,1]}(x)| \leq c \psi_i(x)$ uniformly in $x$, etc.
We now simplify our notation by dropping tilde and keeping in mind that, from now on in this proof, $(x_i)_{i=0}^n$ is the {\em modified} Chebyshev partition.
Hence,
$z_j$ is now the center of $I_{\nu_j}$ (unless $z_j$ is $-1$ or $1$ in which case $z_j$ is, respectively, the left or the right endpoint of $I_{\nu_j}$).

It is convenient to denote
\[
I_* := \left\{ 1\leq i \leq n \st i=\nu_j, 1\leq j \leq M \right\} \andd
 I^* :=\left\{ 1\leq i \leq n \st i\not\in I_* \right\}.
 \]
For each $1\leq i\leq n$,   define $q_i \in\Poly_r$ to be  a polynomial of near best approximation of $f$ on $I_i$ with the weight $w$, \ie
\[
\norm{f-q_i}{\Lp(I_i), w} \leq c E_r(f)_{\Lp(I_i), w} ,
\]
and  define $S_n$ to be a piecewise polynomial function such that $S_n\big|_{I_i} = q_i$, $1\leq i\leq n$.

The following is a crucial observation that follows from \lemp{v} and properties of $w_n$:
\be \label{repl}
w(x) \sim w_n(x)\sim w_n(x_i), \quad \mbox{\rm for each $x\in I_i $ with $i\in  I^*$.}
\ee

Now, using Whitney's inequality we get
\begin{eqnarray*}
\norm{f-S_n}{p, w}^p
& = &
\sum_{i\in I^*}  \int_{I_i} w (x) |f(x)-S_n(x)|^p dx + \sum_{j=1}^M   \int_{I_{\nu_j}} w (x) |f(x)-S_n(x)|^p dx \\
& \leq &
c\sum_{i\in I^*}  w_n (x_i) \int_{I_i} |f(x)-q_i(x)|^p dx +  c\sum_{j=1}^M   E_r(f)_{\Lp(\Z_{\newconst_j,1/n}^j), w}^p   \\
& \leq &
c\sum_{i\in I^*}  w  (x_i) \w_r(f, |I_i|, I_i)_p^p   +  c\sum_{j=1}^M   E_r(f)_{\Lp(\Z_{\newconst_j,1/n}^j), w}^p    \\
& \leq &
c \widetilde \Omega_\varphi^r (f,   \theta/ n)_{\Lp(S), w}^p  +   c\sum_{j=1}^M   E_r(f)_{\Lp(\Z_{\newconst_j,1/n}^j), w}^p ,
\end{eqnarray*}
where $S:= S(1/n) := [-1,1] \setminus \cup_{j=1}^M \Z_{\newconst_j,1/n}^j$. In the last estimate, we  took into account that
  $I_i \subset S(1/n)$, $i\in I^*$.

It is easy to check that $S_n$ can be written as
\[
S_n(x) = q_n(x) + \sum_{i=1}^{n-1} \left[ q_i(x)-q_{i+1}(x) \right] \chi_i(x) ,
\]
and  define
\[
P_n(x) := q_n(x) + \sum_{i=1}^{n-1} \left[ q_i(x)-q_{i+1}(x) \right] T_{i}(x),
\]
where $T_i = T_i(n, \mu)$ are the polynomials from \lem{lem5.1}   with a sufficiently large $\mu$ (we will prescribe it later so   that all restrictions below are satisfied).

\lem{lem5.1} now implies
\begin{eqnarray*}
\norm{  S_n -P_n }{p, w }^p & \leq & \int_{-1}^1 w (x) \left[  \sum_{i=1}^{n-1}  \left| q_i (x)-q_{i+1} (x) \right|  \cdot |\chi_i(x) - T_i(x)|   \right]^p dx \\
& \leq & c
\int_{-1}^1 w (x) \left[  \sum_{i=1}^{n-1}  \norm{q_i -q_{i+1}}{\infty}  \psi_i(x)^\mu   \right]^p dx .
\end{eqnarray*}

Using the Lagrange interpolation formula and \cite[Theorem 4.2.7]{dl}  we have, for all $q \in \Poly_r$ and $0\leq l\leq r-1$,
\be \label{pohidna}
\norm{q^{(l)}}{\infty}  \leq c \psi_i^{-r+l+1} \norm{q^{(l)}}{\C(I_i)}    \leq c \psi_i^{-r+l+1} |I_i|^{-l-1/p} \norm{q}{\Lp(I_i)} ,
\ee
and so it yields (with $l=0$)
\begin{eqnarray*}
\norm{  S_n -P_n }{p, w }^p  & \leq &
c \int_{-1}^1 w (x) \left[  \sum_{i=1}^{n-1}  \norm{q_i -q_{i+1}}{\Lp(I_i)} |I_i|^{-1/p}     \psi_i(x)^{\mu-r+1}    \right]^p dx .
\end{eqnarray*}
Now, if $1\leq p < \infty$, since
$\sum_{i=1}^{n-1} \psi_i(x)^2 \leq c$,
we have by Jensen's inequality
\[
\left( \sum_{i=1}^{n-1} |\gamma_i| \psi_i(x)^2 \right)^p \leq c \sum_{i=1}^{n-1} |\gamma_i|^p \psi_i(x)^2 \leq c \sum_{i=1}^{n-1} |\gamma_i|^p    ,
\]
and if $0<p<1$, then
\[
\left( \sum_{i=1}^{n-1} |\gamma_i| \psi_i(x)^2 \right)^p \leq   \sum_{i=1}^{n-1} |\gamma_i|^p \psi_i(x)^{2p} \leq c \sum_{i=1}^{n-1} |\gamma_i|^p  .
\]

Therefore,
\begin{eqnarray*}
\norm{  S_n -P_n }{p, w }^p  & \leq &
c \int_{-1}^1     \sum_{i=1}^{n-1}  \norm{q_i -q_{i+1}}{\Lp(I_i)}^p |I_i|^{-1}  w (x) \psi_i(x)^{(\mu-r-1)p}   dx \\
& \leq &
c \left(\int_{[-1,1]\setminus \cup_{j=1}^M  I_{\nu_j}} + \sum_{j=1}^M \int_{I_{\nu_j}}     \right) \sum_{i=1}^{n-1}  \norm{q_i -q_{i+1}}{\Lp(I_i)}^p |I_i|^{-1}  w (x) \psi_i(x)^{(\mu-r-1)p}   dx \\
& =:& \II^* + \sum_{j=1}^M \II_j .
\end{eqnarray*}
Hence, since by \ineq{repl}, $w(x) \sim w_n(x) \leq c \psi_i(x)^{-s} w_n(x_i)$, for $x \in [-1,1]\setminus \cup_{j=1}^M I_{\nu_j}$, we have
\begin{eqnarray*}
\II^*
 & \leq &
 c \int_{[-1,1]\setminus \cup_{j=1}^M I_{\nu_j}}     \sum_{i=1}^{n-1}  \norm{q_i -q_{i+1}}{\Lp(I_i)}^p |I_i|^{-1}  w_n (x_i) \psi_i(x)^{(\mu-r-1)p-s}   dx \\
  & \leq &
 c     \sum_{i=1}^{n-1}  \norm{q_i -q_{i+1}}{\Lp(I_i)}^p |I_i|^{-1}  w_n (x_i) \int_{-1}^1 \psi_i(x)^{(\mu-r-1)p-s}   dx \\
   & \leq &
 c     \sum_{i=1}^{n-1}  \norm{q_i -q_{i+1}}{\Lp(I_i)}^p   w_n (x_i) ,
\end{eqnarray*}
if $ (\mu-r-1)p-s \geq 2$, since $\int_{-1}^1 \psi(x)^\a dx \leq c |I_i|$ if $\a\geq 2$.

 Also, for each $1\leq j \leq M$, taking into account    that $|x-x_i|+|I_i| \sim |z_j-x_i|+|I_i|$ and so $\psi_i(x) \sim \psi_i(z_j)$
 uniformly for    $x\in   I_{\nu_j}$,
 we have
\begin{eqnarray*}
 \II_j & \leq  &
 c \sum_{i=1}^{n-1}  \norm{q_i -q_{i+1}}{\Lp(I_i)}^p |I_i|^{-1} \int_{I_{\nu_j}}  w (x) \psi_i(x)^{(\mu-r-1)p}   dx \\
& \leq  &
 c \sum_{i=1}^{n-1}  \norm{q_i -q_{i+1}}{\Lp(I_i)}^p |I_i|^{-1}  \psi_i(z_j)^{(\mu-r-1)p}  \int_{I_{\nu_j}}  w (x)    dx \\
& \leq  &
 c \sum_{i=1}^{n-1}  \norm{q_i -q_{i+1}}{\Lp(I_i)}^p |I_i|^{-1}  \psi_i(z_j)^{(\mu-r-1)p} \dn(z_j) w_n(z_j)  \\
& \leq  &
 c \sum_{i=1}^{n-1}  \norm{q_i -q_{i+1}}{\Lp(I_i)}^p w_n(x_i) |I_i|^{-1}  \psi_i(z_j)^{(\mu-r-1)p-s} \dn(z_j) .
\end{eqnarray*}
 Now, using the inequality $\dnx^2 \leq 4 \dn(y) \left( |x-y|+ \dn(y)\right)$ we have
\begin{eqnarray*}
 |I_i|^{-1}  \psi_i(z_j)^{(\mu-r-1)p-s} \dn(z_j) & \sim & \psi_i(z_j)^{(\mu-r-1)p-s} { \dn(z_j) \over \dn(x_i) }\\
 & \leq &
 c \psi_i(z_j)^{(\mu-r-1)p-s} \left[ {  |x_i-z_j|+ \dn(x_i)  \over \dn(x_i) } \right]^{1/2}\\
 & \sim &
 c \psi_i(z_j)^{(\mu-r-1)p-s-1/2}  \leq c ,
\end{eqnarray*}
provided $(\mu-r-1)p-s-1/2 \geq 0$. Note also that we could alternatively estimate this quantity as follows.
\begin{eqnarray*}
 |I_i|^{-1}  \psi_i(z_j)^{(\mu-r-1)p-s} \dn(z_j)& \sim &  |I_i|^{-1} \int_{I_{\nu_j}} \psi_i(z_j)^{(\mu-r-1)p-s} dx \\
& \sim & |I_i|^{-1} \int_{I_{\nu_j}} \psi_i(x)^{(\mu-r-1)p-s} dx \\
& \leq & c |I_i|^{-1} \int_{-1}^1  \psi_i(x)^{(\mu-r-1)p-s} dx \leq c ,
\end{eqnarray*}
provided $(\mu-r-1)p-s \geq 2$.

Combining the above estimates we conclude that
\[
\norm{  S_n -P_n }{p, w }^p \leq c  \sum_{i=1}^{n-1}  \norm{q_i -q_{i+1}}{\Lp(I_i)}^p   w_n (x_i) .
\]

Now, for each $1\leq j \leq M$, let  $L_j := [z_j-\newconst_j\dn(z_j), z_j-c_0 \newconst_j\dn(z_j)]$ and $R_j := [z_j+c_0 \newconst_j\dn(z_j),z_j+\newconst_j\dn(z_j) ]$
(note that if $z_1=-1$, then $L_1$ is not defined, and if $z_M = 1$, then $R_M$ is not defined, but these intervals are not needed in these cases),
where $c_0 \in (0, 1)$ is a constant that we will choose later (it'll be $0.9$ but we will keep writing ``$c_0$'' in order not to distract from the proof).
Then, $L_j\cup R_j \subset I_{\nu_j}$, $|L_j| \sim |R_j| \sim |I_{\nu_j}|$, and $\dist(L_j, z_j) = \dist(R_j, z_j) = c_0 \newconst_j  \dn(z_j)$, for all $1\leq j \leq M$.

We continue estimating as follows
\begin{eqnarray*}
\norm{  S_n -P_n }{p, w }^p   & \leq &
 c   \left(   \sum_{i, i+1\in I^*} + \sum_{i \in I_*}    +  \sum_{i+1 \in I_*}    \right) \norm{q_i -q_{i+1}}{\Lp(I_i)}^p   w_n (x_i) \\
 & \leq &
 c \sum_{i, i+1\in I^*} \norm{q_i -q_{i+1}}{\Lp(I_i)}^p   w_n (x_i) + c \sum_{j=1}^M   \norm{q_{\nu_j} -q_{\nu_j+1}}{\Lp(I_{\nu_j})}^p   w_n (x_{\nu_j})\\
 && + c \sum_{j=1}^M \norm{q_{\nu_j-1} -q_{\nu_j}}{\Lp(I_{\nu_j-1})}^p   w_n (x_{\nu_j-1}) \\
 & \leq &
 c \sum_{i, i+1\in I^*} \norm{q_i -q_{i+1}}{\Lp(I_i)}^p   w_n (x_i) + c \sum_{j=1}^M   \norm{q_{\nu_j} -q_{\nu_j+1}}{\Lp(L_j)}^p   w_n (x_{\nu_j})\\
 && + c \sum_{j=1}^M \norm{q_{\nu_j-1} -q_{\nu_j}}{\Lp(R_j)}^p   w_n (x_{\nu_j-1}) ,
 \end{eqnarray*}
since  $\norm{q}{\Lp(I)} \sim \norm{q}{\Lp(J)}$, for any polynomial $q\in \Poly_r$ and any intervals $I$ and $J$ of comparable length which are either next to each other or are such that one interval is a subset of the other one.

 Now using \lem{nearbest} (that implies that $q_i$'s are polynomials of near best approximation of $f$ on   intervals which are slightly bigger than $I_i$), Whitney's inequality, \ineq{repl} and
 the fact that $w(x) \sim w_n(x)\sim w_n(x_{\nu_j})$ for each $x\in L_j$ and $w(x) \sim w_n(x)\sim w_n(x_{\nu_j-1})$ for each $x\in R_j$,
 we have
 \begin{eqnarray*}
\norm{  S_n -P_n }{p, w }^p
 & \leq &
 c \sum_{i, i-1 \in I^*}   \norm{ f -q_i}{\Lp(I_i\cup I_{i-1}) }^p  w_n(x_i) +
 c \sum_{j=1}^M    \norm{q_{\nu_j} -f}{\Lp(L_j)}^p  w_n (x_{\nu_j})  \\
 && +
  c \sum_{j=1}^M   \norm{f -q_{\nu_j+1}}{\Lp(L_j)}^p   w_n (x_{\nu_j}) +
 c \sum_{j=1}^M     \norm{q_{\nu_j-1} -f}{\Lp(R_j)}^p w_n (x_{\nu_j-1}) \\
 && +
  c \sum_{j=1}^M   \norm{f -q_{\nu_j}}{\Lp(R_j)}^p    w_n (x_{\nu_j-1}) \\
 & \leq &
c \sum_{i, i-1 \in I^*}  \w_r(f,|I_i\cup I_{i-1}|, I_i\cup I_{i-1})_p^p   w(x_i) +
c \sum_{j=1}^M  \w_r(f, |I_{\nu_j+1} \cup L_j|, I_{\nu_j+1} \cup L_j)_p^p w (x_{\nu_j}) \\
&& + c \sum_{j=1}^M  \w_r(f,|I_{\nu_j-1} \cup R_j| ,I_{\nu_j-1} \cup R_j)_p^p  w (x_{\nu_j-1})
 + c  \sum_{j=1}^M  \norm{f -q_{\nu_j}}{\Lp(I_{\nu_j}), w}^p  \\
 & \leq &
 c \widetilde \Omega_\varphi^r(f,   \theta/n)_{\Lp(\widetilde S), w}^p +  c \sum_{j=1}^M  E_r(f)_{\Lp(\Z_{\newconst_j, 1/n}^j), w}^p ,
\end{eqnarray*}
 where $\widetilde S:= \widetilde S(1/n) := [-1,1] \setminus \cup_{j=1}^M \Z_{c_0\newconst_j, 1/n}^j$ (note that $S(1/n) \subset \widetilde S(1/n)$ and so $\widetilde \Omega_\varphi^r(f,   \theta/n)_{\Lp(S), w}^p \leq \widetilde \Omega_\varphi^r(f,   \theta/n)_{\Lp(\widetilde S), w}^p$).

Now,
\[
P_n^{(\nu)}(x) = p_n^{(\nu)}(x) + \sum_{i=1}^{n-1} \sum_{l=0}^{\nu}  {\nu \choose l}     \left[ q_i^{(l)}(x)-q_{i+1}^{(l)}(x) \right] T_i^{(\nu-l)} (x) ,
\]
and so, for $r\leq \nu \leq \nu_0$ (which guarantees that $p_n^{(\nu)}\equiv 0$), we have using \lem{lem5.1} and estimate \ineq{pohidna}
\begin{eqnarray*}
\norm{ \dn^\nu P_n^{(\nu)}}{p, w }^p & \leq &
\int_{-1}^1 w (x) \dnx^{\nu p}  \left[ \sum_{i=1}^{n-1} \sum_{l=0}^{\nu}  {\nu \choose l}     \left| q_i^{(l)}(x)-q_{i+1}^{(l)}(x) \right| \cdot \left| T_i^{(\nu-l)} (x)\right| \right]^p dx \\
& \leq &
c \int_{-1}^1 w (x)  \dnx^{\nu p} \left[  \sum_{i=1}^{n-1} \sum_{l=0}^{\nu}  \norm{q_i^{(l)} -q_{i+1}^{(l)}}{\infty} |I_i|^{-\nu+l}  \psi_i(x)^\mu   \right]^p dx \\
& \leq &
c \int_{-1}^1 w (x)  \dnx^{\nu p} \left[  \sum_{i=1}^{n-1} \sum_{l=0}^{\nu}  \norm{q_i  -q_{i+1} }{\Lp(I_i)} |I_i|^{-\nu-1/p}  \psi_i(x)^{\mu - r+l+1}   \right]^p dx \\
& \leq &
c \int_{-1}^1 w (x)  \dnx^{\nu p} \left[  \sum_{i=1}^{n-1}    \norm{q_i  -q_{i+1} }{\Lp(I_i)} |I_i|^{-\nu-1/p}  \psi_i(x)^{\mu - r +1}   \right]^p dx \\
& \leq &
c \int_{-1}^1 w (x)  \dnx^{\nu p}    \sum_{i=1}^{n-1}    \norm{q_i  -q_{i+1} }{\Lp(I_i)}^p |I_i|^{-\nu p-1}  \psi_i(x)^{(\mu - r -1)p}   dx .
\end{eqnarray*}
Now, since $\dnx^2 \leq c \dn(x_i) \left( |x-x_i|+ \dn(x_i)\right)$ and $|I_i| \sim \dn(x_i)$, we have
\begin{eqnarray*}
\lefteqn{ \norm{ \dn^\nu P_n^{(\nu)}}{p, w }^p }\\
& \leq &
c \int_{-1}^1    w (x)  \sum_{i=1}^{n-1}    \norm{q_i  -q_{i+1} }{\Lp(I_i)}^p   \left[ \dn(x_i) \left( |x-x_i|+ \dn(x_i)\right) \right]^{\nu p/2}      |I_i|^{-\nu p-1} \psi_i(x)^{(\mu - r -1)p }   dx \\
& \leq &
c \int_{-1}^1   w (x)  \sum_{i=1}^{n-1}    \norm{q_i  -q_{i+1} }{\Lp(I_i)}^p   |I_i|^{-1}   \psi_i(x)^{(\mu - r -1-\nu/2)p }   dx ,
 \end{eqnarray*}
and exactly the same sequence of inequalities as above (only the power of $\psi_i$ is different) yields
\[
\norm{ \dn^\nu P_n^{(\nu)}}{p, w }  \leq c   \widetilde \Omega_\varphi^r(f,   \theta/n)_{\Lp(\widetilde S), w}^p +  c \sum_{j=1}^M  E_r(f)_{\Lp(\Z_{\newconst_j, 1/n}^j), w}^p   ,
\]
provided $(\mu - r -1-\nu_0/2)p-s \geq 2$.

Thus, if we pick $\mu = \mu(r, \nu_0, p, s)$ so that this (the most restrictive in this proof) inequality is satisfied
then, for each $m\in\N$,
we constructed a polynomial $ P_m$ of degree $< n_0 m$ with some $n_0\in\N$ depending only on $r$, $\nu_0$, $p$ and the doubling constant of the weight $w$, such that
\[
\max\left\{\norm{ \dm^\nu   P_m^{(\nu)}}{p, w },  \norm{f-  P_m}{p, w} \right\} \leq c   \widetilde \Omega_\varphi^r(f,   \theta/m)_{\Lp(\widetilde S(1/m)), w}^p +  c \sum_{j=1}^M  E_r(f)_{\Lp(\Z_{\newconst_j, 1/m}^j), w}^p .
 \]

Suppose now that $n\geq D n_0 =:N$, where $D$ is a natural number $\geq 10$ that  will be picked in a moment.  Then there exists $m\in\N$ such that $m n_0 \leq n < (m+1)n_0$ (note that $m\geq D$ and so $n_0 \leq  n/m \leq (1+1/D) n_0$).
  Then the polynomial $P_m$ is of degree $< n_0 m \leq n$ (\ie $  P_m \in \Poly_n$).

Now, we need to pick   $\theta$, $\newconst_j$'s, $c_0$ and $D$  so that
\be \label{hugemess}
 \widetilde \Omega_\varphi^r(f,   \theta/m)_{\Lp(\widetilde S(1/m)), w}^p +  c \sum_{j=1}^M  E_r(f)_{\Lp(\Z_{\newconst_j, 1/m}^j), w}^p   \leq c \widetilde \w_\varphi^r(f, A, \vartheta/n)_{p, w} .
\ee
This will complete the proof since $\dm(x) \sim \dnx$.

The estimate \ineq{hugemess} is satisfied if, in particular, for   $1\leq j \leq M$,
\[
 \Z_{A, \vartheta/n}^j \subset \Z_{c_0 \newconst_j, 1/m}^j ,                  \quad \Z_{\newconst_j, 1/m}^j \subset \Z_{2A, \vartheta/n}^j \andd \theta/m \leq \vartheta/n
\]
(see properties of the moduli in Section~\ref{moduli}).
We pick $\theta$ so that $\theta \leq  \vartheta/(2 n_0)$, and to finish the proof we need to make sure that the following holds:
\be \label{makesure}
c_0 \newconst_j \dm(z_j) \geq A \rho(\vartheta/n, z_j) \andd  \newconst_j \dm(z_j) \leq 2A \rho(\vartheta/n, z_j), \quad 1\leq j \leq M .
\ee
Recall that $\newconst_j$ is assumed to be $\leq 1/10$, and that it cannot depend on $m$ or $n$ (but can depend on $n_0$).
We also note that we can assume that $\vartheta$ is small since
\[
\widetilde \w_\varphi^r(f, A, \vartheta_1/n)_{p, w} \leq c \widetilde \w_\varphi^r(f, A, \vartheta_2/n)_{p, w}, \quad \mbox{\rm if }\; \vartheta_1 \leq \vartheta_2 .
\]
So we assume that $\vartheta \leq 1$ is such that it guarantees that $\newconst_j \leq 1/10$ (see the estimates below). Alternatively, we can guarantee this by letting $n_0$ be sufficiently large.

Hence, if $z_j = \pm 1$ the inequalities in \ineq{makesure} become
\[
A \vartheta^2  \leq c_0 \newconst_j {n^2 \over m^2}  \andd   \newconst_j {n^2 \over m^2} \leq 2A \vartheta^2  ,
\]
and recalling that $n_0 \leq  n/m \leq (1+1/D) n_0$, we   now pick $\newconst_j$ so that
\[
{A \vartheta^2 \over c_0 n_0^2} \leq \newconst_j  \leq {2 A \vartheta^2 \over (1+1/D)^2 n_0^2} .
\]
For example, with $c_0:= 0.9$ we set $\newconst_j :=  A \vartheta^2/(0.9 n_0^2)$ (recall that $D\geq 10$).

We now let $D\geq 10$   be so large that $D\geq 10/\varphi(z_j)$ for all $1\leq j \leq M$, for which $z_j \neq \pm 1$ (so, clearly, $D$ depends only on the weight $w$). Recalling that $n\geq m \geq D$, this implies that, if $z_j \neq \pm 1$, then
 \[
 \varphi(z_j)/m \leq   \rho(1/m, z_j) \leq 1.1  \varphi(z_j)/m \andd \vartheta\varphi(z_j)/n \leq   \rho(\vartheta/n, z_j) \leq 1.1 \vartheta\varphi(z_j)/n .
 \]
 Therefore, to guarantee that the inequalities in  \ineq{makesure} hold it is sufficient to pick $\newconst_j$ so  that
\[
{1.1 A \vartheta \over c_0} \leq \newconst_j {n \over m} \andd 1.1 \newconst_j {n \over m} \leq 2A \vartheta ,
\]
which, in turn, follows from
 \[
{1.1 A \vartheta \over c_0 n_0} \leq \newconst_j    \leq { 2A \vartheta \over 1.1 (1+1/D)n_0 } .
\]
Now, recall that we already picked  $c_0 = 0.9$, and let
\[
\newconst_j  := {1.1 A \vartheta \over 0.9 n_0},
\]
for all $1\leq j \leq M$ such that $z_j\neq \pm 1$.
\end{proof}

\sect{Remez and Markov-Bernstein type theorems and applications} \label{remezmb}

Most results in this section are based on a well known idea to
 use  Remez type results to
go back and forth between $\varphi(x)$ and $\varphi(x)+1/n$ in various estimates involving polynomials   and on the fact that
$\norm{P_n}{p, w} \sim \norm{P_n}{p, w_n}$ for polynomials from $\Poly_n$ (G. Mastroianni and V. Totik deserve most credit for this observation).
Note that most of them are given for general doubling weights without the requirement that they belong to $\W(\Z)$ (but see a comment following the statement of \cor{cordw}).

\subsection{Remez type theorems and applications}

We start with he following crucial lemma that states that the norms of polynomials of degree $< n$ are essentially the same irrespectively of whether the weight $w$ or the weight $w_n$ is used (where $w$ is a doubling weight).

\begin{lemma} \label{auxlemma}
Let $w$ be a doubling weight on $[-1,1]$. Then for every $0<  p < \infty$ there is a constant $c_0$ depending only on $p$ and the doubling constant of $w$ such that, for every polynomial $P_n\in\Poly_n$,
\[
c_0^{-1} \norm{P_n}{p, w} \leq \norm{P_n}{p, w_n} \leq c_0  \norm{P_n}{p, w} .
\]
\end{lemma}

In the case $1\leq p <\infty$, this is  \cite[Theorem 7.2]{mt2000}.
It is obtained in \cite{mt2000} as a corollary of an analogous result for trigonometric polynomials (see \cite[Theorem 3.1]{mt2000}) with a method that does not depend on whether or not $p$ is greater or less than $1$.
Since the result for trigonometric polynomials holds for all $0<p<\infty$ (see \cite[Theorem 2.1]{e}), we conclude that \lem{auxlemma} is valid.

The following Remez inequality for doubling weights holds.

\begin{theorem}[\mbox{\cites{e, mt2000}}]
Let $W$ be a $2\pi$-periodic function which is a doubling weight on $[0, 2\pi]$, and let $0<p<\infty$ be arbitrary. Then there is a constant $C>0$ depending only on $p$ and on the doubling constant of $W$ so that if $T_n$ is a trigonometric polynomial of degree at most $n$ and $E$ is a measurable subset of $[0, 2\pi]$ of measure at most $\Lambda/n$, $1\leq \Lambda\leq n$, that is a union of intervals of length at least $c/n$, then
\[
\int_{-\pi}^\pi |T_n(u)|^p W(u) du  \leq \left( {C \over c }\right)^{\Lambda} \int_{[0, 2\pi]\setminus E} |T_n(u)|^p W(u) du .
\]
\end{theorem}

%
%
%

The following is a corollary for algebraic polynomials (see \cite{mt2000} in the case $1\leq p <\infty$, the case $0<p<1$ is analogous).


\begin{corollary} \label{cordw}
Let  $w$ be a doubling weight and $0 < p <  \infty$.
If $E \subset [-1,1]$ is a union of at most $K$ intervals  and $ \int_E (1-x^2)^{-1/2} dx   \leq \Lambda/n$, $\Lambda \leq n$, then for each $p_n\in\Poly_n$, we have
\[
\int_{-1}^1 |p_n(x)|^p w(x)\, dx \leq C \int_{[-1,1]\setminus E} |p_n(x)|^p w(x) \, dx ,
\]
where the constant $C$ depends only on $\Lambda$, $K$, $p$ and the doubling constant of $w$.
\end{corollary}

We note that there is a simple proof showing that \cor{cordw} is satisfied for doubling weights from the class $\W(\Z)$. This follows from the usual unweighted Remez inequality (\ie \cor{cordw} with $w\equiv 1$) and the fact that $w_n(x) \sim \Q_n(x)^p$, where $0<p<\infty$ and $\Q_n \in \Poly_n$ (see \cite[(7.34)-(7.36)]{mt2000} or \cite[Theorem 4.1]{k-weighted}).

%

Indeed, suppose that $E \subset [-1,1]$ is a union of at most $K$ intervals  and $ \int_E (1-x^2)^{-1/2} dx   \leq c/n$. We enlarge $E$ to $E \cup \widetilde E$, where
 $\widetilde E := \Z_{1, 1/n} =  [-1,1]\cap \cup_{j=1}^M [z_j -  \dn(z_j), z_j +   \dn(z_j)]$ and note that 
\[
\int_{\widetilde E} (1-x^2)^{-1/2} dx \leq  \sum_{j=1}^M \int_{z_j -  \dn(z_j)}^{z_j +   \dn(z_j)}(1-x^2)^{-1/2} dx \leq c/n .
\]
Then, using \lem{auxlemma} we have
\begin{eqnarray*}
\norm{P_n}{p, w} &\sim&  \norm{P_n}{p, w_n} \leq c \norm{P_n\Q_n}{p } \leq c \norm{P_n\Q_n}{\Lp([-1,1]\setminus (E\cup \widetilde E))}
 \leq   c \norm{P_n}{\Lp([-1,1]\setminus (E\cup \widetilde E)), w_n}\\
 & \leq& c \norm{P_n}{\Lp([-1,1]\setminus (E\cup \widetilde E)), w}    \leq c \norm{P_n}{\Lp([-1,1]\setminus E), w} ,
\end{eqnarray*}
since $w\sim w_n$ on $[-1,1]\setminus \widetilde E$ by \lemp{v}.

One of the applications of \cor{cordw} is the following result which is quite useful in the proofs.

\begin{theorem} \label{mainauxthm}
Let $w$ be a doubling weight, $0<p<\infty$, $n\in\N$, $0\leq \mu \leq n$. Then, for any  $P_n \in \Poly_n$,
\be \label{mainauxineq1}
 \norm{\varphi ^\mu P_n}{p, w} \sim \norm{\varphi ^\mu P_n}{p, w_n}
\ee
and
\be \label{mainauxineq2}
\norm{\lambda_n^\mu P_n}{p, w}  \sim \norm{\lambda_n^\mu P_n}{p, w_n}    ,
\ee
where $\lambda_n(x) := \max\left\{  \sqrt{1-x^2} , 1/n \right\}$,
and
the equivalence constants depend only on $p$  and the doubling constant of $w$, and are independent of $\mu$.
\end{theorem}

 \begin{proof} The idea used in this proof is well known. Since $w\sim w_n$ and $\lambda_n \sim \varphi$ in the ``middle'' of $[-1,1]$ the quantities are equivalent by the Remez type result allowing us to replace $[-1,1]$ by $[-1+n^{-2}, 1-n^{-2}]$. We have to be careful with the constants though making sure that they do not depend on $\mu$.

 We start with the equivalence \ineq{mainauxineq1}.
Note that if $\mu$ is an even integer, then this equivalence immediately follows from \lem{auxlemma}  since $\varphi^\mu P_n \in \Poly_{n+\mu} \subset \Poly_{2n}$ and $w_n \sim w_{2n}$.
It is now clear how to proceed. We let $m:= 2\lfloor \mu/2 \rfloor$. Then $m$ is an even integer such that $\mu-2 < m \leq \mu$ (note that $m=0$ if $\mu<2$), and
$Q_{n+m} := \varphi^m P_n \in \Poly_{n+m} \subset \Poly_{2n}$.

Since $w$ is a doubling weight, then $w \varphi^{\gamma p}$, $\gamma>0$,  is also a doubling weight (with a doubling constant depending on $\lceil \gamma\rceil $, $p$ and the doubling constant of $w$)
and (see also \cite[Lemma 4.5 and p. 65]{mt2000})
 \[
 (w \varphi^{\gamma p})_n \sim w_n \varphi_n^{\gamma p},
 \]
where   $\varphi_n(x) \sim \varphi(x) + 1/n$, and
the equivalence constants depend on $\lceil \gamma\rceil $, $p$  and the doubling constant of $w$.

Hence, denoting $S_n :=[-1+n^{-2}, 1-n^{-2}]$,  $\eta := \mu-m$, noting that $0\leq \eta <2 $ (and so $\lceil \eta \rceil$ is either $1$ or $2$ allowing us to replace constant that depend on $\lceil \eta \rceil$ by those independent of $\eta$), and using \lem{auxlemma} and \cor{cordw}  we have
\begin{eqnarray*}
\norm{\varphi ^\mu P_n}{p, w} &=& \norm{\varphi^{\eta}  Q_{n+m}}{p, w} = \norm{  Q_{n+m}}{p, w \varphi^{\eta p}} \sim \norm{  Q_{n+m}}{p, (w \varphi^{\eta p})_n}  \sim \norm{  Q_{n+m}}{\Lp(S_n), (w \varphi^{\eta p})_n}         \\
& \sim & \norm{  Q_{n+m}}{\Lp(S_n), w_n \varphi_n^{\eta p} }
\sim \norm{  Q_{n+m}}{\Lp(S_n), w_n \varphi^{\eta p} } .
\end{eqnarray*}
Now, since the weight $w_n \varphi^{\eta p}$ is doubling with the doubling constant depending only on the doubling constant of $w$  and $p$, we can continue as follows.
\begin{eqnarray*}
 \norm{  Q_{n+m}}{\Lp(S_n), w_n \varphi^{\eta p} } \sim  \norm{  Q_{n+m}}{p, w_n \varphi^{\eta p} }
=
  \norm{ \varphi^{\eta}  Q_{n+m}}{p, w_n } = \norm{   \varphi^\mu   P_{n}}{p, w_n } .
\end{eqnarray*}
Note that none of the constants in the equivalences above depend on $\mu$. This completes the proof of \ineq{mainauxineq1}.

\medskip

Now, let $\E_n := \left\{ x \st \sqrt{1-x^2} \leq 1/n \right\}$ and note that $\lambda_n(x) =1/n$ if $x\in \E_n$, and $\lambda_n(x) =\varphi(x)$ if $x\in [-1,1]\setminus \E_n$.
Using \ineq{mainauxineq1} we have
\begin{eqnarray*}
2^{\min\{0, 1-1/p\}} \norm{\lambda_n^\mu P_n}{p, w}  & \leq &
 \norm{\lambda_n^\mu P_n}{\Lp(\E_n), w} + \norm{\lambda_n^\mu P_n}{\Lp([-1,1]\setminus \E_n), w} \\
 & = &
 n^{- \mu}  \norm{ P_n}{\Lp(\E_n), w} +   \norm{\varphi^\mu P_n}{\Lp([-1,1]\setminus \E_n), w} \\
 & \leq &
 n^{- \mu}  \norm{ P_n}{p, w} +   \norm{\varphi^\mu P_n}{p, w} \\
  & \leq &
c_0   \left( n^{- \mu}   \norm{ P_n}{p, w_n} +   \norm{\varphi^\mu P_n}{p, w_n}\right) \\
 & \leq &
2c_0  \norm{\lambda_n^\mu P_n}{p, w_n} .
\end{eqnarray*}
In the other direction, the sequence of inequalities is exactly the same (switching $w$ and $w_n$).
This verifies \ineq{mainauxineq2}.

\end{proof}

If we allow constants to depend on $\mu$, then we have the following result.

\begin{corollary} \label{newjust1}
Let $w$ be a doubling weight, $0<p<\infty$,  $n\in\N$ and $\mu \geq 0$. Then, for any  $P_n \in \Poly_n$,
\[
\norm{\varphi_n^\mu P_n}{p, w} \sim    \norm{\varphi^\mu P_n}{p, w} \sim \norm{\varphi^\mu P_n}{p, w_n}   \sim \norm{\varphi_n^\mu P_n}{p, w_n} ,
\]
where the equivalence constants depend only on $p$, $\mu$   and the doubling constant of $w$.
\end{corollary}

\begin{proof}
Since $\lambda_n(x) \sim \varphi_n(x)\sim \varphi(x) + 1/n$
  and $\varphi(x) \leq \varphi(x) + 1/n \sim \varphi_n(x)$, we immediately get from \thm{mainauxthm}
\[
\norm{\varphi^\mu P_n}{p, w} \sim \norm{\varphi^\mu P_n}{p, w_n} \leq c \norm{\varphi_n^\mu P_n}{p, w_n}\sim \norm{\varphi_n^\mu P_n}{p, w} .
\]
The following sequence finishes the proof:
\begin{eqnarray*}
\norm{\varphi^\mu P_n}{p, w} & = & \norm{ P_n}{p, w\varphi^{\mu p}} \sim \norm{ P_n}{p, (w\varphi^{\mu p})_n}  \sim \norm{ P_n}{p, w_n \varphi_n^{\mu p} } = \norm{\varphi_n^{\mu} P_n}{p, w_n  }.
\end{eqnarray*}
\end{proof}

%
%
%
%

 \subsection{Markov-Bernstein  type theorems}

In this subsection, we continue with the applications of the results presented in the first part of this section and discuss several Markov-Bernstein estimates for doubling weights.

We note that the following theorem can be obtained from  \cite[Theorem 4.1]{mt2000} and \cite[Theorem 3.1]{e} (Markov-Bernstein estimate for trigonometric polynomials)  with the same proof as that of   \cite[Theorem 7.3, (7.10) and (7.12)]{mt2000}. However, we  provide an alternative proof using the equivalence results from the previous section.


\begin{theorem} \label{thm5.5}
Let $w$ be a doubling weight, $0<p<\infty$ and  $r\in\N$.  Then, for all $n\in\N$ and $P_n\in\Poly_n$,
\[
  n^{-r} \norm{\varphi^r P_n^{(r)}}{p, w}    \sim  n^{-r} \norm{\varphi^r P_n^{(r)}}{p, w_n}  \sim \norm{\dn^r P_n^{(r)}}{p, w_n}   \sim   \norm{\dn^r P_n^{(r)}}{p, w} \leq c   \norm{ P_n }{p, w} \sim  \norm{ P_n }{p, w_n} ,
\]
where the constant $c$ and the equivalence constants depend only on $r$, $p$  and the doubling constant of $w$.
\end{theorem}

%

\begin{proof} 
The statement of the lemma is an immediate consequence of \cor{newjust1} and the following estimate
(see \cite[Lemma 6.1]{k-weighted}, for example)
\[
   \norm{\dn^r P_n^{(r)} }{p, w_n} \leq c \norm{P_n}{p, w_n} ,
\]
where   the constant $c$ depends only on $r$, $p$ and the doubling constant of $w$.
\end{proof}

In the proof of inverse results for $0<p<1$ we need to know how the constants in Markov-Bernstein estimates depend on the order of derivatives.

We start with the following result that was proved in \cite{k-weighted} (see Corollaries 6.4 and 6.6 there).

\begin{lemma} \label{oldcors}
Let $w$ be a doubling weight and $0<p<1$. Then, for all
  $n,r\in\N$ and $l\in\N_0$ such that  $l\leq r \leq n-1$,   and $P_n \in \Poly_n$,
\[
\norm{\delta_n^{r } P_n^{(r)}}{p,w_n}  \leq     (c_*)^{r-l} {r! \over l!}   \norm{\delta_n^{l} P_n^{(l)}}{p,w_n}
\]
and
\[
\norm{\varphi^{r } P_n^{(r)}}{p,w_n}  \leq     (c_*)^{r-l} {r! \over l!} n^{r-l}  \norm{\varphi^{l} P_n^{(l)}}{p,w_n} ,
\]
where $\delta_n(x) := \max\left\{  \sqrt{1-x^2}/ n , 1/n^2 \right\}$, and
the constant $c_*$ depends only on $p$ and the doubling constant of $w$.
\end{lemma}

We remark that if we are not interested in the exact dependance of the constants on $l$ (the order of the lower derivative in the estimates), then
the first estimate in \lem{oldcors} and \cor{newjust1} imply the following (weaker) analog of the second estimate in \lem{oldcors} which actually would have been sufficient for our purposes:
\[
\norm{\varphi^{r } P_n^{(r)}}{p,w_n} \leq n^r \norm{\delta_n^{r } P_n^{(r)}}{p,w_n} \leq
c (c_*)^{r} r! n^r \norm{\delta_n^{l} P_n^{(l)}}{p,w_n}  \leq c (c_*)^{r} r! n^{r-l} \norm{\varphi^{l} P_n^{(l)}}{p,w_n} ,
\]
where $c$ is allowed to depend on $l$ in addition to $p$ and the doubling constant of $w$.


Taking into account \lem{mainauxthm} and observing that $\delta_n(x) = \lambda_n(x)/n$  we immediately get the following corollary (in order not to overcomplicate the notation we incorporate the extra constant  into $c_*$, \ie we emphasize once again that constants $c_*$ in different statements are different).

\begin{corollary} \label{cor2.9}
Let $w$ be a doubling weight and $0<p<1$. Then, for all
  $n,r\in\N$ and $l\in\N_0$ such that  $l\leq r \leq n-1$,   and $P_n \in \Poly_n$,
\[
\norm{\delta_n^{r } P_n^{(r)}}{p,w }  \leq     (c_*)^{r-l} {r! \over l!}   \norm{\delta_n^{l} P_n^{(l)}}{p,w }
\]
and
\[
\norm{\varphi^{r } P_n^{(r)}}{p,w}  \leq     (c_*)^{r-l} {r! \over l!} n^{r-l}  \norm{\varphi^{l} P_n^{(l)}}{p,w} ,
\]
where
the constant  $c_*$ depends  only on $p$ and the doubling constant of $w$.
\end{corollary}

Now, taking into account that   $\delta_n(x) \leq \dnx \leq 2 \delta_n(x)$, this immediately implies the following.

\begin{corollary} \label{improv}
Let $w$ be a doubling weight and $0<p<1$. Then, for all
  $n,r\in\N$ and $l\in\N_0$ such that  $l\leq r \leq n-1$,   and $P_n \in \Poly_n$,
\[
\norm{\dn^{r } P_n^{(r)}}{p,w }  \leq    2^l  (c_*)^{r-l} {r! \over l!}   \norm{\dn^{l} P_n^{(l)}}{p,w }  ,
\]
where
the constant  $c_*$ depends  only on $p$ and the doubling constant of $w$.
\end{corollary}

\sect{Two crucial auxiliary lemmas} \label{cruciallemmas}

In the case $1\leq p <\infty$, we have the following lemma.

\begin{lemma} \label{lem8.5j}
Let $w$ be a doubling weight, $1\leq p < \infty$ and $A >0$.  Then   for any
$n, r\in\N$, $I := \Z_{A, 1/n}^j$, and any polynomials $Q_n\in\Poly_n$ and $q_r \in \Poly_r$ satisfying  $Q_n^{(\nu)}(z_j)=q_r^{(\nu)}(z_j)$, $0\leq \nu\leq r-1$, the following inequality holds
\[
\norm{Q_n-q_r}{\Lp(I), w}  \leq c n^{-r}     \norm{\varphi^r  Q_n^{(r)} }{p, w} ,
\]
where the constant $c$ depends only on $r$, $p$, $A$ and the doubling constant of $w$.
\end{lemma}

\begin{remark}
Using the same proof it is possible to show that, for any $f$ such that $f^{(r-1)}\in\AC(I)$,
\[
E_r(f)_{\Lp(I), w} \leq c  \norm{\dn^r  f^{(r)} }{\Lp(I), w_n} .
\]
At the same time, $w_n$ on the right-hand side of this estimate cannot be replaced with $w$ since, otherwise, together with \lem{lem7.2} and \thm{jacksonthm} we would get the estimate
$E_n(f)_{p,w} \leq c  \norm{\dn^r  f^{(r)} }{p, w}$ which is not valid for all doubling weights (see \cite[Example 3.5]{mt1998}). In fact, even the estimate
$E_r(f)_{\Lp(I), w} \leq c  \norm{\dn^r  f^{(r)} }{p, w}$ is invalid in general.
\end{remark}

\begin{proof}[Proof of \lem{lem8.5j}]
The proof is rather straightforward and relies on Taylor's theorem
(see \eg \cite[Proposition 4.1]{dbmr}). However, since it is short and works for all doubling weights,  we sketch it below for completeness.
Denote  $z:=z_j$, and note that $(Q_n - q_r)^{(\nu)}(z)=0$, $0\leq \nu\leq r-1$, and that we can assume that $n\geq r+1$.
Using Taylor's theorem with the integral remainder 
we have
\[
Q_n(x) - q_r(x)  = {1 \over (r-1)!} \int_z^x  (x-u)^{r-1}  Q_n^{(r)}(u) du .
\]
Hence,  using H\"{o}lder's inequality (with $1/p+1/p'=1$) we have
\begin{eqnarray*}
  \norm{Q_n-q_r}{\Lp(I), w}^p    & \leq &
  \int_I w(x)  \left[\int_{[z,x]} |x-u|^{r-1}  |Q_n^{(r)}(u)|   du  \right]^p  dx    \\
& \leq &
\int_I   w(x)  \left(\int_{[z,x]} |x-u|^{(r-1)p'} du \right)^{p/p'}
 \int_{[z,x]} |Q_n^{(r)}(u)|^p   du    dx    \\
& \leq &
\int_I   w(x) |x-z|^{rp-1}     \int_{[z,x]} |Q_n^{(r)}(u)|^p   du    dx    \\
& \leq & (A \dn(z))^{rp-1} \norm{Q_n^{(r)}}{\Lp(I)}^p \int_I w(x) dx .
\end{eqnarray*}
Now, using the fact that $w(I) \leq c w\left([z-\dn(z), z+\dn(z)]\right)$ with $c$ depending only on $A$ and the doubling constant of $w$, and the fact that
$\dnx \sim \dn(z)$ and $w_n(x) \sim w_n(z)$, for $x\in I$, we have
\begin{eqnarray*}
  \norm{Q_n-q_r}{\Lp(I), w}^p \leq c  w_n(z)  \norm{\dn^r Q_n^{(r)}}{\Lp(I)}^p  \leq c \norm{\dn^r Q_n^{(r)}}{\Lp(I), w_n}^p
  \leq c \norm{\dn^r Q_n^{(r)}}{p, w_n}^p
  \leq c n^{-r} \norm{\varphi^r Q_n^{(r)}}{p, w}^p ,
\end{eqnarray*}
 where the last estimate follows from \cor{newjust1} .
\end{proof}

If $0<p<1$, we no longer can use H\"{o}lder's inequality in a straightforward way, and so it takes much more effort  to get an analog of \lem{lem8.5j}. If there is a simple  proof of the following lemma, we were unable to find it.

\begin{lemma} \label{crucialpless1}
Let $w$ be a doubling weight and $0<p<1$. Then there exists  a positive constant $\theta\leq 1$ depending only on $p$ and the doubling constant of $w$ such that, for
$n, r\in\N$, $I := \Z_{\theta, 1/n}^j$, and any polynomials $Q_n\in\Poly_n$ and $q_r \in \Poly_r$ satisfying  $Q_n^{(\nu)}(z_j)=q_r^{(\nu)}(z_j)$, $0\leq \nu\leq r-1$, the following inequality holds
\[
\norm{Q_n-q_r}{\Lp(I), w}  \leq c   n^{-r} \norm{ \varphi^r Q_n^{(r)} }{p, w} ,
\]
where the constant $c$ depends only on $r$, $p$ and the doubling constant of $w$.
\end{lemma}

%

\begin{proof}
We use the approach from \cite[Section 6]{dlub}.
Denote $g := Q_n - q_r$ and $z:=z_j$, and note that $g^{(\nu)}(z)=0$, $0\leq \nu\leq r-1$, and $g^{(r)} = Q_n^{(r)}$.
Using Taylor's theorem with the integral remainder 
we have
\[
g(x) = {1 \over (r-1)!} \int_z^x  (x-u)^{r-1}  g^{(r)}(u) du .
\]
Hence,  

\begin{eqnarray*}
\lefteqn{ \norm{Q_n-q_r}{\Lp(I), w}^p   =  \norm{g}{\Lp(I), w}^p = \int_I |g(x)|^p w(x) dx
  \leq
  \int_I \left|\int_z^x (x-u)^{r-1}  g^{(r)}(u)  w(x)^{1/p} du  \right|^p  dx }  \\
& \leq &
  \int_I \left|\int_z^x  \left|(x-u)^{r-1}  g^{(r)}(u) \right|^{1-p}  w(x)^{-1+1/p}
  \times \left|(x-u)^{r-1}  g^{(r)}(u)  \right|^{p}  w(x)      du   \right|^p   dx \\
&\leq &
 \int_I \norm{ \left|(x-u)^{r-1}  g^{(r)}(u)  \right|^{1-p}    w(x)^{-1+1/p}   }{\L_\infty[z,x]}^p
   \times \left|\int_z^x  \left|(x-u)^{r-1}  g^{(r)}(u) \right|^{p}   w(x)     du   \right|^p   dx .
\end{eqnarray*}
Now, using H\"{o}lder's inequality with $\sigma_1 = 1/(1-p)$ and $\sigma_2 = 1/p$ (note that $1/\sigma_1 + 1/\sigma_2 = 1$) we have
\begin{eqnarray*}
\norm{g}{\Lp(I), w}^p & \leq &
 \left[\int_I \norm{ \left|(x-u)^{r-1}  g^{(r)}(u)  \right|^{1-p}   w(x)^{-1+1/p} }{\L_\infty[z,x]}^{p/(1-p)}  dx \right]^{1/\sigma_1} \\
 && \times     \left[   \int_I  \left| \int_z^x  \left|(x-u)^{r-1}  g^{(r)}(u)  \right|^{p}  w(x)     du  \right| dx \right]^{1/\sigma_2} \\
& \leq &
  \left[\int_I \norm{ \left|(x-u)^{r-1}  g^{(r)}(u)  \right|^{p}  w(x)    }{\L_\infty[z,x]}   dx \right]^{1-p} \\
 && \times     \left[   \int_I  \left| \int_z^x  \left|(x-u)^{r-1}  g^{(r)}(u)  \right|^{p}  w(x)      du \right|  dx \right]^{p} =:   T_1^p \times T_2^p .
\end{eqnarray*}

To estimate $T_2$ we recall that $[z,x] := [x,z]$ if $x<z$ and write
\begin{eqnarray*}
T_2 & = &
\int_I  w(x)   \int_{[z,x]}   |x-u|^{(r-1)p}  |g^{(r)}(u)|^p        du   dx \\
& \leq &
\int_I  \left|  g^{(r)}(u)\right|^p du \int_I w(x) |x-z|^{(r-1)p}    dx \\
& \leq &
(\theta \dn(z))^{(r-1)p} \norm{g^{(r)}}{\Lp(I)}^p  \int_I w(x) dx \\
& \leq &
(\theta \dn(z))^{(r-1)p} \dn(z) w_n(z)  \norm{g^{(r)}}{\Lp(I)}^p .
\end{eqnarray*}
Since  $w_n(x)\sim w_n(z)$ and $\dn(x) \sim \dn(z)$, $x\in I$, this implies
\[
T_2 \leq c \dn(z)^{1-p} \norm{\dn^r g^{(r)}}{\Lp(I), w_n}^p .
\]

Now, we need to estimate
\begin{eqnarray*}
T_1^{p/(1-p)} &=&   \int_I  \norm{ \left|(x-u)^{r-1}  g^{(r)}(u) \right|^{p}  w(x)   }{\L_\infty[z,x]}   dx   .
\end{eqnarray*}

For $u$ between $z$ and $x$ we have
\begin{eqnarray*}
|x-u|^{(r-1)p}  |g^{(r)}(u)|^p &= & |x-u|^{(r-1)p}  \left| \sum_{\nu=r}^{n-1} {g^{(\nu)}(x) \over (\nu-r)!} (u-x)^{\nu-r} \right|^p\\
& \leq&
 \sum_{\nu=r}^{n-1} |g^{(\nu)}(x)|^p \left[ { |x-z|^{\nu-1} \over (\nu-r)!}  \right]^p ,
\end{eqnarray*}
and so
\begin{eqnarray*}
T_1^{p/(1-p)} &\leq &
\int_I  w(x)   \sum_{\nu=r}^{n-1} |g^{(\nu)}(x)|^p \left[ { |x-z|^{\nu-1} \over (\nu-r)!}  \right]^p    dx \\
& \leq &
c\int_I    w(x) \sum_{\nu=r}^{n-1} | \dn(x)^\nu g^{(\nu)}(x)|^p \left[ { |x-z|^{\nu-1} \over \dn(z)^\nu (\nu-r)!}  \right]^p        dx \\
&\leq &
c\sum_{\nu=r}^{n-1} \left[ { (\theta \dn(z))^{\nu-1} \over \dn(z)^\nu (\nu-r)!}  \right]^p   \norm{\dn^\nu g^{(\nu)} }{\Lp(I), w}^p \\
&\leq &
c\sum_{\nu=r}^{n-1}   \left[ { \theta^{\nu-1} \over \dn(z)  (\nu-r)!}  \right]^p   \norm{\dn^\nu g^{(\nu)} }{p, w}^p  .
\end{eqnarray*}
We now use \cor{improv} to conclude
\begin{eqnarray*}
T_1^{p/(1-p)} &\leq &
c  \sum_{\nu=r}^{n-1}   \left[ { \theta^{\nu-1} \over \dn(z)  (\nu-r)!}       2^r (c_*)^{\nu-r} { \nu! \over r!}     \right]^p   \norm{\dn^r g^{(r)} }{p, w}^p \\
  & \leq &
 c \dn(z)^{-p} \norm{\dn^r g^{(r)} }{p, w}^p   \sum_{\nu=r}^{\infty}   (\theta c_*)^{\nu p}     \left[{\nu \choose r} \right]^p      \\
& \leq &
 c \dn(z)^{-p} \norm{\dn^r g^{(r)} }{p, w}^p ,
\end{eqnarray*}
provided $\theta c_*  \leq 1/2$.
 Therefore,
\[
T_1  \leq
c \dn(z)^{p-1}   \norm{\dn^r g^{(r)} }{p, w }^{1-p} .
\]
Combining estimates of $T_1$ and $T_2$, we have
\[
\norm{g}{\Lp(I)} \leq  c   \norm{\dn^r g^{(r)}}{\Lp(I), w_n}^p \norm{\dn^r g^{(r)} }{p, w }^{1-p} \leq c
   \norm{\dn^r g^{(r)}}{p, w_n}^p       \norm{\dn^r g^{(r)} }{p , w}^{1-p}  \leq c   n^{-r} \norm{\varphi^r g^{(r)} }{p , w}   ,
\]
noting that the last estimate  immediately follows from \cor{newjust1}.
\end{proof}

\sect{Preliminary results for inverse theorems} \label{prelinv}

\begin{lemma} \label{lem7.1}
If $w$ is a doubling weight from the class $\W(\Z)$, $ 0< p< \infty$, $f\in\Lp^w$,
  $r\in\N$, and $A,   t >0$,  then
\[
 \w_\varphi^r(f, A,   t)_{p, w}    \leq     c \norm{f}{p, w }  ,
\]
where $c$ depends only on  $r$, $p$, $A$  and the weight $w$.
\end{lemma}

\begin{proof}
First of all, it is clear that
\[
\sum_{j=1}^M  E_r(f)_{\Lp(\Z_{2A,t}^j), w}          
\leq \sum_{j=1}^M  \norm{f}{\Lp(\Z_{2A,t}^j), w} \leq M \norm{f}{p, w}.
\]
Now, recall that
$\Delta_{h\varphi(x)}^r(f,x, \I_{A,h}) = 0$ if $x\not\in S_{A,h} \subset \Dom_{rh/2}$, where
\[
S_{A,h} := \left\{ x \st [x-rh\varphi(x)/2,  x+rh\varphi(x)/2 ] \subset \I_{A, h}\right\}
\]
and
\[
\Dom_{rh/2} := \left\{ x \st x\neq \pm 1 \andd x \pm rh \varphi(x)/2 \in [-1,1] \right\}  = \left\{ x \st |x| \leq (4-r^2h^2)/(4+r^2h^2) \right\},
\]
and so
 \begin{eqnarray*}
 \Omega_\varphi^r(f, A, t)_{p, w} & = & \sup_{0<h\leq t} \norm{\Delta_{h\varphi}^r(f)}{\Lp(S_{A, h}),w} .
\end{eqnarray*}

Let $h\in (0, t]$ be fixed, $x \in S_{A,h}$
and denote $y_i(x) := x+(i-r/2)h \varphi(x)$, $0\leq i\leq r$.
Then, $[x, y_i(x)] \subset \I_{A, h}$ and $|x-y_i(x)| \leq (r/2) \rho(h, x)$, and so \lemp{iv} implies that $w(x)\sim w(y_i(x))$.

Now, taking into account that  $1/2 \leq y_i'(x) \leq 3/2$, $x\in \Dom_{rh/2}$, we have
\begin{eqnarray*}
\norm{\Delta_{h\varphi}^r(f)}{\Lp(S_{A,h}),w}^p & \leq &    \int_{S_{A,h}} w  (x)
\left( \sum_{i=0}^r {r \choose i} |f(x+(i-r/2)h \varphi(x)) | \right)^p \, dx\\
& \leq &
c    \sum_{i=0}^r \int_{S_{A,h}} w (y_i(x))
\left|f(y_i(x))   \right|^p \, dx \\
& \leq &
c    \int_{-1}^1 w (y) \left|f(y)  \right|^p \, dy  \leq c \norm{f}{p, w}^p  .
\end{eqnarray*}
\end{proof}

\begin{lemma} \label{lem7.2}
Let $w$ be a doubling weight from the class $\W(\Z)$, $1\leq p <\infty$
  $n,r\in\N$ and  $A, t >0$. If $f$ is such that $f^{(r-1)}\in \AC_\loc \left( (-1,1)\setminus \Z  \right)$
  and $\norm{\varphi^r f^{(r)}}{p, w} < \infty$, then
\[
\Omega_\varphi^r(f, A, t)_{p, w} \leq c t^r \norm{  \varphi^r f^{(r)}}{p, w} ,
\]
where $c$ depends only on  $r$, $A$,   $p$ and the weight $w$.
\end{lemma}


\begin{proof}
Recall that
 \begin{eqnarray*}
 \Omega_\varphi^r(f, A, t)_{p, w} & = & \sup_{0<h\leq t} \norm{\Delta_{h\varphi}^r(f)}{\Lp(S_{A, h}),w} ,
\end{eqnarray*}
where
\[
S_{A,h} := \left\{ x \st[x-rh\varphi(x)/2,  x+rh\varphi(x)/2 ] \subset \I_{A, h}\right\} .
\]
Since $\I_{A, h}$ has at most $M+1$ components, it is sufficient (and necessary) to verify the lemma for each of them.
We have two different types of components: when a component is ``in the middle'' of $[-1,1]$, \ie
\[
J_{A, h}^j  := [z_j+A \rho(h, z_j), z_{j+1}-A \rho(h, z_{j+1})], \quad \mbox{\rm where } \; 1\leq j \leq M-1 ,
\]
and when a component is near the endpoints of $[-1,1]$. Note that there is a component of this type only   when $z_1 \neq -1$ and $z_{M}\neq 1$.
More precisely, define
\[
J_{A, h}^0  := [-1, z_{1}-A \rho(h, z_{1})] \quad \mbox{\rm if }\; z_1 \neq -1
\]
and
\[
J_{A, h}^M := [z_M+A \rho(h, z_j), 1] \quad \mbox{\rm if }\;  z_M \neq  1 .
\]

%

Recall that
$\Delta_{h\varphi(x)}^r(f,x, \I_{A,h}) = 0$ if $x\in J_{A, h}^j$ and  $[x-rh\varphi(x)/2,  x+rh\varphi(x)/2 ] \not\subset J_{A, h}^j$, and so we also denote
\[
S_{A, h}^j := \left\{ x \st [x-rh\varphi(x)/2,  x+rh\varphi(x)/2 ] \subset J_{A, h}^j \right\}, \quad 0\leq j \leq M .
\]

Suppose now that $1\leq p <\infty$ and let $h\in (0, t]$ be fixed.
Since $f$ has the $(r-1)$st locally absolutely continuous derivative inside each $S_{A, h}^j$, we have for any $x\in S_{A, h}^j$
\[
  \Delta_{h\varphi(x)}^r(f, x) = \int_{-h\varphi(x)/2}^{h\varphi(x)/2} \dots \int_{-h\varphi(x)/2}^{h\varphi(x)/2} f^{(r)}(x + t_1 + \dots + t_r) dt_r \dots dt_1 ,
\]
and, by \lemp{iv}, $w(x) \sim w(u)$, for $u\in [x-rh\varphi(x)/2,  x+rh\varphi(x)/2 ]$.

Therefore,
\begin{eqnarray*}
\lefteqn{ \left( \int_{S_{A, h}^j}  w(x) |  \Delta_{h\varphi(x)}^r(f, x)|^p dx\right)^{1/p} }\\
& \leq &
\left(  \int_{S_{A, h}^j}  \left[ \int_{-h\varphi(x)/2}^{h\varphi(x)/2} \dots \int_{-h\varphi(x)/2}^{h\varphi(x)/2} w(x)^{1/p}   |f^{(r)}(x + t_1 + \dots + t_r   )| dt_r \dots dt_1\right]^p dx  \right)^{1/p} \\
& \leq &
c\left(  \int_{S_{A, h}^j}  \left[ \int_{-h\varphi(x)/2}^{h\varphi(x)/2} \dots \int_{-h\varphi(x)/2}^{h\varphi(x)/2} w(x + t_1 + \dots + t_r)^{1/p}   |f^{(r)}(x + t_1 + \dots + t_r   )| dt_r \dots dt_1\right]^p dx  \right)^{1/p} .
\end{eqnarray*}
By H\"{o}lder's inequality, for each $u$ satisfying $[x+u - h\varphi(x)/2, x+u + h\varphi(x)/2] \subset S_{A, h}^j$, we have
\begin{eqnarray*}
\int_{-h\varphi(x)/2}^{h\varphi(x)/2} w(x + u + t_r)^{1/p}   |f^{(r)}(x + u  + t_r   )| dt_r & = &
\int_{x+u-h\varphi(x)/2}^{x+u+h\varphi(x)/2} w(v)^{1/p}   |f^{(r)}(v)| dv \\
& \leq &
\norm{ w^{1/p} \varphi^r f^{(r)}}{\Lp(\A(x,u))}  \norm{\varphi^{-r}}{\L_{p'} (\A(x,u))} ,
\end{eqnarray*}
where $1/p+1/p'=1$ and
\[
\A(x,u) := \left[ x+u-h\varphi(x)/2, x+u+h\varphi(x)/2 \right] .
\]
The needed estimate now follows from
\begin{eqnarray} \label{mess1}
\lefteqn{ \int_{S_{A, h}^j}  \left[ \int_{-h\varphi(x)/2}^{h\varphi(x)/2} \dots \int_{-h\varphi(x)/2}^{h\varphi(x)/2}  \norm{\varphi^{-r}}{\L_{p'} (\A(x,t_1+\dots+t_{r-1}))} \right. } \\ \nonumber
&& \left. \times
\norm{ w^{1/p} \varphi^r f^{(r)}}{\Lp(\A(x,t_1+\dots+t_{r-1}))}     dt_{r-1}  \dots dt_1\right]^p dx  \leq c h^{rp}\norm{ w^{1/p} \varphi^r f^{(r)}}{p}^p , \quad 1\leq p <\infty .
\end{eqnarray}
Note that, in the case $r=1$,   \ineq{mess1} is understood as
\begin{eqnarray} \label{mess3}
 \int_{S_{A, h}^j}      \norm{\varphi^{-1}}{\L_{p'} (\A(x,0))}^p
\norm{ w^{1/p} \varphi  f'}{\Lp(\A(x,0))}^p   dx  \leq c h^{p}\norm{ w^{1/p} \varphi f'}{p}^p , \quad 1\leq p < \infty .
\end{eqnarray}

Estimates \ineq{mess1} and \ineq{mess3}  are proved in exactly the same way as \cite[(4.2)-(4.4)]{kls-ca}.
\end{proof}


The following lemma can be proved using exactly the same sequence of estimates that were used to prove \cite[Lemma 6.9]{k-weighted} with the only difference that the second estimate of \cor{cor2.9} should be used instead of
\cite[Corollary 6.6]{k-weighted}.

\begin{lemma}
Let $w$ be a doubling weight, $0<p<1$  and $n,r\in\N$.
Then,   there exists a positive constant $\ccc$ depending only on  $r$,  $p$ and  the doubling constant of $w$,
 such that, for any $P_n \in \Poly_n$ and $0<h \leq \ccc/n$,
 \[  
\left( 1/2 \right)^{1/p}  h^{r} \norm{\varphi^r P_n^{(r)}}{p, w}    \leq   \norm{\Delta_{h\varphi}^r (P_n)}{p, w}\leq (3/2)^{1/p} h^{r} \norm{\varphi^r P_n^{(r)}}{p, w} .
\]
 \end{lemma}

Taking into account that
 \begin{eqnarray*}
 \Omega_\varphi^r(P_n, A, t)_{p, w} &= & \sup_{0<h\leq t} \norm{\Delta_{h\varphi}^r(P_n)}{\Lp(\I_{A, h}),w}  \leq \sup_{0<h\leq t} \norm{\Delta_{h\varphi}^r(P_n)}{p,w}
\end{eqnarray*}
 we immediately get the following corollary.

\begin{corollary} \label{cor7.4}
Let $w$ be a doubling weight, $0<p<1$ and $n,r\in\N$.
Then,   there exists a positive constant $\ccc \leq 1$ depending only on  $r$,  $p$ and  the doubling constant of $w$,
 such that, for any $P_n \in \Poly_n$, $A>0$ and $0<t \leq \ccc/n$,
\[
\Omega_\varphi^r(P_n, A, t)_{p, w} \leq   n^{-r} \norm{\varphi^r P_n^{(r)}}{p, w} .
\]
\end{corollary}

%
%
%

 \sect{Inverse theorem for $1\leq p < \infty$} \label{invbig}

\begin{theorem} \label{conversethm}
Suppose that $w$ is  a doubling weight from the class $\W(\Z)$,
 $r\in\N$, $1\leq p <  \infty$, and $f\in\Lp^w$.
Then
\[
 \w_\varphi^r(f, A,  n^{-1})_{p, w}
  \leq  c n^{-r}   \sum_{k=1}^{n}    k^{r-1  }     E_{k}(f)_{p,w } ,
\]
where
the constant $c$ depends only on $r$, $p$, $A$    and the  weight $w$.
\end{theorem}

\begin{proof}
Let $P_n^* \in \Poly_n$ denote a polynomial of (near) best approximation to $f$ with weight $w$, \ie
\[
c \norm{f-P_n^*}{p,w} \leq   \inf_{P_n\in\Poly_n} \norm{f-P_n}{p,w} =: E_{n}(f)_{p,w} .
\]
We let $N\in\N$ be such that $2^N \leq n < 2^{N+1}$.
To estimate $\Omega_\varphi^r(f, A, n^{-1})_{p, w}$,
using \lem{lem7.1} we have
\begin{eqnarray*}
\Omega_\varphi^r(f, A, n^{-1})_{p, w} & \leq & \Omega_\varphi^r (f, A, 2^{-N})_{p, w} \\
& \leq & \Omega_\varphi^r (f - P_{2^N}^*, A, 2^{-N})_{p, w} + \Omega_\varphi^r (P_{2^N}^*, A, 2^{-N})_{p, w} \\
& \leq & c \norm{f - P_{2^N}^*}{p, w } + \Omega_\varphi^r (P_{2^N}^*, A, 2^{-N})_{p, w} \\
& \leq &
 c E_{2^N}(f)_{p,w} + \Omega_\varphi^r (P_{2^N}^*, A, 2^{-N})_{p, w}.
\end{eqnarray*}
Now, using
\be \label{decom}
P_{2^N}^* =  P_1^* + \sum_{i=0}^{N-1} (P_{2^{i+1}}^* - P_{2^{i}}^*)
\ee
as well as \lem{lem7.2}  we have
\begin{eqnarray*}
\Omega_\varphi^r (P_{2^N}^*, A, 2^{-N})_{p, w}  &\leq &
 \sum_{i=0}^{N-1} \Omega_\varphi^r  \left( P_{2^{i+1}}^* - P_{2^{i}}^*, A, 2^{-N}\right)_{p, w }
 \leq
 c 2^{-Nr } \sum_{i=0}^{N-1}    \norm{ \varphi^r \left(P_{2^{i+1}}^* - P_{2^{i}}^*\right)^{(r)}}{p,w} .
\end{eqnarray*}

Now, for each $1\leq j\leq M$, taking into account that $\Z_{2A,t_1}^j \subset \Z_{2A,t_2}^j$ if $t_1 \leq t_2$, we have
\begin{eqnarray*}
E_r(f)_{\Lp(\Z_{2A,1/n}^j), w}    &\leq&  \inf_{q  \in\Poly_r} \norm{f-q }{\Lp(\Z_{2A,2^{-N}}^j), w} \\
 & \leq &   \norm{f-P_{2^N}^*}{\Lp(\Z_{2A,2^{-N}}^j), w} +   \inf_{q  \in\Poly_r} \norm{P_{2^N}^*-q }{\Lp(\Z_{2A,2^{-N}}^j), w} \\
 & \leq & c E_{2^N}(f)_{p, w} +   \norm{P_{2^N}^*-q_r(P_{2^N}^*) }{\Lp(\Z_{2A,2^{-N}}^j), w},
\end{eqnarray*}
 where $q_r(g)$ denotes the Taylor polynomial of degree $<r$ at $z_j$ for $g$.
Using \ineq{decom} again,
    noting that
\be \label{extrataylor}
q_r(P_{2^N}^*) =  P_1^* + \sum_{i=0}^{N-1} q_r(P_{2^{i+1}}^* - P_{2^{i}}^*),
\ee
and taking \lem{lem8.5j} into account we have
\begin{eqnarray*}
\norm{P_{2^N}^*-q_r(P_{2^N}^*) }{\Lp(\Z_{2A,2^{-N}}^j), w} & \leq & \sum_{i=0}^{N-1} \norm{ (P_{2^{i+1}}^* - P_{2^{i}}^*) - q_r(P_{2^{i+1}}^* - P_{2^{i}}^*)}{\Lp(\Z_{2A,2^{-N}}^j), w} \\
& \leq & c \sum_{i=0}^{N-1}  2^{-Nr} \norm{\varphi^r (P_{2^{i+1}}^* - P_{2^{i}}^*)^{(r)}}{p, w} .
\end{eqnarray*}
Hence,
\[
 \w_\varphi^r(f, A,  n^{-1})_{p, w} \leq c E_{2^N}(f)_{p, w} + c 2^{-Nr } \sum_{i=0}^{N-1}    \norm{ \varphi^r \left(P_{2^{i+1}}^* - P_{2^{i}}^*\right)^{(r)}}{p,w}.
\]

Now, using    \thm{thm5.5} we have

\begin{eqnarray*}
 \w_\varphi^r(f, A,   n^{-1})_{p, w} &\leq & c E_{2^N}(f)_{p, w} +
  c 2^{-Nr } \sum_{i=0}^{N-1}   2^{ir }  \norm{ P_{2^{i+1}}^* - P_{2^{i}}^*  }{p, w}\\
 &\leq&
  c 2^{-Nr } \sum_{i=0}^{N}   2^{ir }    E_{2^i}(f)_{p,w } \\
& \leq &  c n^{-r}  \left(E_{1}(f)_{p,w} +  \sum_{i=1}^{N}  \sum_{k=2^{i-1}+1}^{2^{i}} k^{r-1 }     E_{k}(f)_{p,w} \right) \\
 & \leq &
c n^{-r}  \sum_{k=1}^{n}    k^{r-1}     E_{k}(f)_{p,w}
\end{eqnarray*}
with all constants $c$ depending only on $r$, $p$, $A$  and the weight $w$.
\end{proof}

\sect{Inverse theorem for $0<p<1$} \label{invsmall}

 \begin{theorem} \label{conversethmpless1}
Suppose that $w$ is  a doubling weight from the class $\W(\Z)$, and let
 $r\in\N$, $A >0$, $0< p <1$, and $f\in\Lp^w$.
Then there exists a positive constant $\ccc \leq 1$ depending only on $p$, $r$, $A$ and the doubling constant of $w$, and such that
\[
 \w_\varphi^r(f, A,   \ccc n^{-1})_{p, w}^p
  \leq  c n^{-rp}   \sum_{k=1}^{n}    k^{rp-1  }     E_{k}(f)_{p,w }^p ,
\]
where
the constant $c$ depends only on $r$, $p$, $A$   and the weight    $w$.
\end{theorem}

\begin{proof} The method of the proof is standard and well known (see \cite{djl} or \cite{k-weighted}).
With the same notation as in the proof of \thm{conversethm} (\ie $P_n^*$ is a polynomial of (near) best weighted approximation to $f$ and $2^N \leq n < 2^{N+1}$), we have using \lem{lem7.1} (note that we will be putting restrictions on $\ccc$ as we go along)
\begin{eqnarray*}
\Omega_\varphi^r(f, A, \vartheta n^{-1})_{p, w}^p & \leq & 
 c E_{2^N}(f)_{p,w}^p + \Omega_\varphi^r (P_{2^N}^*, A, \vartheta 2^{-N})_{p, w}^p
\end{eqnarray*}
and, using \ineq{decom},
\begin{eqnarray*}
\Omega_\varphi^r (P_{2^N}^*, A, \vartheta 2^{-N})_{p, w}^p  &\leq &
 \sum_{i=0}^{N-1} \Omega_\varphi^r  \left( P_{2^{i+1}}^* - P_{2^{i}}^*, A, \vartheta 2^{-N}\right)_{p, w }^p.
\end{eqnarray*}
\lem{lem7.2} can no longer be used, and so we employ \cor{cor7.4} (we assume that the current constant $\ccc \leq 1$  is not bigger than $\ccc$ from \cor{cor7.4})
 which implies
\begin{eqnarray*}
\Omega_\varphi^r (P_{2^N}^*, A, \vartheta 2^{-N})_{p, w}^p
& \leq &
 2^{-N rp}  \sum_{i=0}^{N-1}  \norm{\varphi^r  \left( P_{2^{i+1}}^* - P_{2^{i}}^* \right)^{(r)}}{p, w }^p.
\end{eqnarray*}

For each $1\leq j \leq M$,   recalling that $q_r(g)$ denotes the Taylor polynomial of degree $<r$ at $z_j$ for $g$,  we have
\begin{eqnarray*}
E_r(f)_{\Lp(\Z_{2A,\ccc/n}^j), w}^p
   &\leq&  \inf_{q  \in\Poly_r} \norm{f-q }{\Lp(\Z_{2A,\ccc 2^{-N}}^j), w}^p \\
 & \leq &  c E_{2^N}(f)_{p, w}^p +   \norm{P_{2^N}^*-q_r(P_{2^N}^*) }{\Lp(\Z_{2A, \ccc 2^{-N}}^j), w}^p .
\end{eqnarray*}
Now, we make sure that $\ccc$ is so small that
 \[
 \Z_{2A, \ccc 2^{-N}}^j  \subset  \Z_{\theta,   2^{-N}}^j , \quad 1\leq j \leq M,
 \]
 where $\theta$ is the constant from \lem{crucialpless1}. This is achieved if $\ccc \leq \theta/(2A)$.
 Therefore, \lem{crucialpless1} implies
\[
\norm{ Q  - q_r(Q) }{\Lp(  \Z_{2A, \ccc 2^{-N}}^j ), w} \leq c 2^{-Nr} \norm{\varphi^r Q^{(r)}}{p, w} , \quad \mbox{\rm for any }\; Q\in \Poly_{2^N} ,
\]
with $c$ depending only on $r$, $p$ and the doubling constant of $w$.
Hence, using \ineq{decom} and \ineq{extrataylor} we obtain
\begin{eqnarray*}
\norm{P_{2^N}^*-q_r(P_{2^N}^*) }{\Lp(\Z_{2A, \ccc 2^{-N}}^j), w}^p & \leq &
\sum_{i=0}^{N-1} \norm{ (P_{2^{i+1}}^* - P_{2^{i}}^*) - q_r(P_{2^{i+1}}^* - P_{2^{i}}^*)}{\Lp(\Z_{2A,\ccc 2^{-N}}^j), w}^p \\
& \leq & c \sum_{i=0}^{N-1}  2^{-Nrp} \norm{\varphi^r (P_{2^{i+1}}^* - P_{2^{i}}^*)^{(r)}}{p, w}^p .
\end{eqnarray*}
Therefore,
\[
 \w_\varphi^r(f, A,  \ccc n^{-1})_{p, w}^p \leq c E_{2^N}(f)_{p, w}^p + c 2^{-Nrp } \sum_{i=0}^{N-1}    \norm{ \varphi^r \left(P_{2^{i+1}}^* - P_{2^{i}}^*\right)^{(r)}}{p,w}^p.
\]
Now, using   \thm{thm5.5}, similarly to  the case $1\leq p < \infty$, we get
\begin{eqnarray*}
 \w_\varphi^r(f, A,   \ccc n^{-1})_{p, w}^p &\leq & c E_{2^N}(f)_{p, w}^p +
  c 2^{-Nrp } \sum_{i=0}^{N-1}   2^{irp }  \norm{ P_{2^{i+1}}^* - P_{2^{i}}^*  }{p, w}^p\\
 &\leq&
  c 2^{-Nrp } \sum_{i=0}^{N}   2^{irp }    E_{2^i}(f)_{p,w }^p \\
 & \leq &
c n^{-rp}  \sum_{k=1}^{n}    k^{rp-1}     E_{k}(f)_{p,w}^p
\end{eqnarray*}
with all constants $c$ depending only on $r$, $p$  and the weight $w$.
\end{proof}

 \sect{Equivalence of moduli and Realization functionals} \label{real}

 Let $w$ be a doubling weight from the class $\W(\Z)$, $r \in\N$, $0<p<\infty$   and $f\in\Lp^w$.


We define the following  realization functionals as follows
 \[
R_{r,\varphi} (f, t,   \Poly_n)_{p, w} := \inf_{P_n\in\Poly_n} \left( \norm{f-P_n}{p, w} + t^r \norm{\varphi^r P_n^{(r)}}{p, w} \right) .
 \]
Clearly, $R_{r,\varphi} (f, t_1,   \Poly_n)_{p, w} \sim R_{r,\varphi} (f, t_2,   \Poly_n)_{p, w}$ if $t_1 \sim t_2$.

 \thm{jacksonthm} implies that, for every $n\geq N$ (with $N$ depending only on $r$,   $p$ and  $w$),
$\ccc >0$ and $A>0$,  there exists a polynomial $P_n \in\Poly_n$ such that
\be \label{kf1}
R_{r,\varphi} (f, 1/n, \Poly_n)_{p, w}   \leq c \widetilde \w_\varphi^r(f, A,   \ccc/n)_{p, w },
\ee
where constants $c$  depend only on  $r$,  $p$, $\ccc$, $A$ and the weight $w$.

 \lem{lem7.1}  implies that, for any $0<p<\infty$, $f\in\Lp^w$,  $r\in\N$ and $A,  t >0$, and any $g\in\Lp^w$,
 \begin{eqnarray} 
 \Omega_\varphi^r(f, A,  t)_{p, w}    &\leq&  c \Omega_\varphi^r(f-g, A,  t)_{p, w}  + c \Omega_\varphi^r(g, A,  t)_{p, w} \\ \nonumber
 & \leq & c \norm{f-g}{p,w} + c \Omega_\varphi^r(g, A,  t)_{p, w} ,
\end{eqnarray}
where $c$ depends only on  $r$, $p$, $A$  and the weight $w$.

Now, in the case $1\leq p <\infty$, \lem{lem7.2} additionally yields that,
if  $g$ is such that $g^{(r-1)}\in \AC_\loc \left( (-1,1)\setminus \Z  \right)$
  and $\norm{\varphi^r g^{(r)}}{p, w} < \infty$, then
 \begin{eqnarray*}
 \Omega_\varphi^r(f, A, t)_{p, w}
 & \leq &
 c \norm{f-g}{p,w} + c t^r \norm{  \varphi^r g^{(r)}}{p, w}.
\end{eqnarray*}
This,
in particular, implies that, if $1\leq p <\infty$, then for any $n\in\N$, $\ccc >0$, $A >0$ and $0<t \leq \ccc/n$,
 \begin{eqnarray} \label{kf3}
 \Omega_\varphi^r(f, A, t)_{p, w}
 & \leq &
 c \norm{f-P_n}{p,w} + c n^{-r} \norm{  \varphi^r P_n^{(r)}}{p, w},
\end{eqnarray}
where $c$ depends only on  $r$, $p$, $\ccc$,  $A$  and the weight $w$.

If we use \cor{cor7.4} instead of \lem{lem7.2} then we conclude that \ineq{kf3} is valid if $0<p<1$ as well, but now $0<\ccc\leq 1$ is some fixed constant that depends on $r$, $p$ and the doubling constant of $w$.

Now, using \lem{lem8.5j} we have, for $1\leq p <\infty$, any $\ccc>0$ and $0<t\leq \ccc/n$ (taking into account that $\Z_{2A,t}^j \subset \Z_{2A, \ccc/n}^j \subset \Z_{2A\ccc\max\{\ccc,1\}, 1/n}^j$),
\begin{eqnarray} \label{kf4}
\sum_{j=1}^M \inf_{q\in\Poly_r} \norm{f- q}{\Lp(\Z_{2A,t}^j), w}
 & \leq & c \norm{f-P_n}{p,w}   + \sum_{j=1}^M \inf_{q\in\Poly_r} \norm{P_n - q}{\Lp(\Z_{2A\ccc\max\{\ccc,1\}, 1/n}^j), w}\\ \nonumber
 & \leq &
 c \norm{f-P_n}{p,w} +    c n^{-r}  \norm{\varphi^r  P_n^{(r)} }{p, w},
\end{eqnarray}
where constants $c$ depend on $r$, $p$, $A$, $\ccc$ and the doubling constant of $w$.

In the case $0<p<1$, using \lem{crucialpless1} we conclude that there exists $0<\ccc \leq 1$ depending only on $p$, $A$ and the doubling constant of $w$ such that, for $0<t\leq \ccc/n$, \ineq{kf4} is satisfied with constants $c$ that depend on $r$, $p$, $A$, and the doubling constant of $w$. Note that this follows from the observation that
$\Z_{2A,t}^j \subset \Z_{2A, \ccc/n}^j \subset \Z_{2A\ccc, 1/n}^j \subset \Z_{\theta, 1/n}^j$, where $\theta$ is the constant from the statement of \lem{crucialpless1} and    $\ccc := \min\{\theta/(2A), 1\}$.

Hence, we actually verified the validity of the following two corollaries.
First, \ineq{kf1}, \ineq{kf3} and \ineq{kf4} yield the following result.

\begin{corollary} \label{corr1}
Let $w$ be a doubling weight from the class $\W(\Z)$, $r \in\N$, $1\leq p<\infty$   and $f\in\Lp^w$.
Then there exists a constant $N\in\N$ depending on $r$, $p$ and the weight $w$ such that, for any $\ccc_2 \geq \ccc_1 >0$,  $n\geq N$,   $\ccc_1/n \leq t \leq \ccc_2/n$, and $A>0$, we have
\[
R_{r,\varphi} (f, t, \Poly_n)_{p, w}     \sim     \widetilde \w_\varphi^r(f, A,   t)_{p, w } \sim   \w_\varphi^r(f, A,   t)_{p, w } .
\]
\end{corollary}

 \cor{corr1} implies, in particular,  that $\w_\varphi^r(f, A_1,  t_1)_{p, w } \sim \w_\varphi^r(f, A_2,  t_2)_{p, w }$ if $A_1 \sim A_2$ and $t_1 \sim t_2$.

In the case $0<p<1$, we have

\begin{corollary} \label{corr2}
Let $w$ be a doubling weight from the class $\W(\Z)$, $r \in\N$, $0<p<1$, $A>0$,   and $f\in\Lp^w$.
Then there exist  $N\in\N$  depending on $r$, $p$ and the weight $w$, and $\ccc>0$ depending on  $r$, $p$, $A$, and the doubling constant of $w$,  such that, for any $\ccc_1 \in (0, \ccc]$,   $n\geq N$,   $\ccc_1/n \leq t \leq \ccc/n$,   we have
\[
R_{r,\varphi} (f, t, \Poly_n)_{p, w}     \sim     \widetilde \w_\varphi^r(f, A,   t)_{p, w } \sim   \w_\varphi^r(f, A,   t)_{p, w } .
\]
\end{corollary}

\cor{corr2} implies that, for  $A_1, A_2>0$, $A_1 \sim A_2$,  there exists $t_0 >0$ such that   $\w_\varphi^r(f, A_1,  t_1)_{p, w } \sim \w_\varphi^r(f, A_2,  t_2)_{p, w }$  for  $0< t_1, t_2 \leq t_0$ such that $t_1 \sim t_2$.

\sect{Appendix} \label{appendix}

\begin{lemma}  \label{nearbest}
 Suppose that $w$ is a doubling weight from the class $\W(\Z)$, $0<p<\infty$, $f\in\Lp^w$,
 and suppose that intervals $I$ and $J$ are such that $I\subset J \subset [-1,1]$ and $|J|\leq  c_0 |I| $.
 Then, for any $r\in\N$, if
 $q \in\Poly_r$ is a polynomial of near best approximation to $f$ on $I$ in the $\Lp$ (quasi)norm with weight $w$, \ie
  \[
\norm{f-q}{\Lp(I), w} \leq c_1 E_r(f)_{\Lp(I), w} ,
\]
then $q$ is also a polynomial of near best approximation to $f$ on $J$. In other words,
  \[
\norm{f-q}{\Lp(J), w} \leq c E_r(f)_{\Lp(J), w} ,
\]
where the constant $c$ depends only on $p$, $c_0$, $c_1$ and the weight $w$.
 \end{lemma}

\begin{proof} First, we assume that $|I| \leq \D/2$, and so $I$ may contain at most one $z_j$ from $\Z$.
Now, we denote by $a$  the midpoint of $I$ and
 let $n\in\N$ be such that
 \[
\rho_{n+1}(a) < |I|/1000 \leq \dn(a) .
 \]
 Then
 $ |I|/1000 \leq \dn(a) \leq  |I|/250$.

We recall again that   $\dnx\leq |I_i| \leq 5 \dnx$ for $x\in I_i$, and $|I_{i\pm 1}| \leq 3 |I_i|$. Hence, if $a \in I_\nu$, for some $\nu$, then
\[
\sum_{i=0}^2 |I_{\nu-i}| \leq (1+3+9)|I_\nu| = 13|I_\nu|\leq 65 \dn(a) < |I|/2 ,
\]
and so $I_{\nu-1} \cup I_{\nu-2} \subset I$.
Similarly,
\[
\sum_{i=0}^2 |I_{\nu+i}| \leq (1+3+9)|I_\nu| = 13|I_\nu|\leq 65 \dn(a) < |I|/2 ,
\]
and so $I_{\nu+1} \cup I_{\nu+2} \subset I$.

%

In other words, $I$ contains at least $5$ adjacent  intervals $I_{\nu+i}$, $i=2,1,0,-1,-2$.  Since $I$ contains at most one $z_j$, we now can pick one of these $5$ intervals in such a way that there is another interval $I_i$ between it and $z_j$ (if $I$ does not contain any $z_j$'s, we pick one of the intervals ``in the middle'' of $I$, for example $I_{\nu}$). Suppose that the interval that we picked is $I_\mu$.
Then,
\[
|I| \geq |I_\mu| \geq |I_\nu|/9 \geq \dn(a)/9 \geq |I|/9000,
\]
\ie $|I_\mu|\sim |I|$.
Also,  $I_\mu \subset \I_{c, 1/n}$ with some absolute constant $c$, and \lemp{iv} implies that $w(x) \sim w (y)$, for $x,y\in I_\mu$, with equivalence constants depending only on $w$.

%
%

Suppose now  that $\widetilde q$ is a polynomial of near best approximation of $f$ on $J$, \ie
 \[
\norm{f-\widetilde q}{\Lp(J), w} \leq c E_r(f)_{\Lp(J), w} .
\]
Then, taking into account that $|I_\mu| \sim |I| \sim |J|$ and using properties of doubling weights (see \cite[Lemma 2.1(vi) and Lemma 7.1]{mt2000}, for example), we have
\begin{eqnarray*}
\norm{\widetilde q -q}{\Lp(J), w}^p  &=&  \int_J w(x) |\widetilde q(x) -q(x)|^p dx \leq \norm{\widetilde q - q}{\C(J)}^p \int_J w(x) dx \\
& \leq &
c \norm{\widetilde q - q}{\C(I_\mu)}^p \int_{I_\mu} w(x) dx
  \leq
c|I_\mu|^{-1} \norm{\widetilde q - q}{\Lp(I_\mu)}^p \int_{I_\mu} w(x_\mu) dx \\
&\leq& c \int_{I_\mu} |\widetilde q(x) -q(x)|^p w(x_\mu) dx
\leq c \norm{\widetilde q -q}{\Lp(I_\mu), w}^p \\
& \leq & c \norm{\widetilde q -q}{\Lp(I), w}^p .
\end{eqnarray*}

Therefore,
\begin{eqnarray*}
\norm{f-q}{\Lp(J), w} &\leq& c \norm{f-\widetilde q}{\Lp(J), w}+ c \norm{\widetilde q -q}{\Lp(J), w} \\
& \leq &
c \norm{f-\widetilde q}{\Lp(J), w} + c \norm{\widetilde q -q}{\Lp(I), w} \\
& \leq &
c \norm{f-\widetilde q}{\Lp(J), w} + c \norm{\widetilde q -f}{\Lp(I), w}    +  c \norm{  f -q}{\Lp(I), w}  \\
& \leq &
c \norm{f-\widetilde q}{\Lp(J), w}     +  c \norm{  f -q}{\Lp(I), w}  \\
& \leq &
c E_r(f)_{\Lp(J), w} + c E_r(f)_{\Lp(I), w} \\
& \leq &
c E_r(f)_{\Lp(J), w} ,
\end{eqnarray*}
 and the proof is complete if $|I| \leq \D/2$.

If $|I| > \D/2$, then $|I|\sim |J|\sim 1$, and we take $n\in\N$ to be such $I$ contains at least $4M+4$ intervals $I_i$. Then $I$ contains   $4$ adjacent intervals $I_i$ not containing any points from $\Z$, and we can use the same argument as above.
\end{proof}

\end{document}